\documentclass{amsart}
\newenvironment{dedication}
        {\vspace{6ex}\begin{quotation}\begin{center}\begin{em}}
        {\par\end{em}\end{center}\end{quotation}}
\usepackage{amssymb}
\usepackage{tikz}
\usepackage{epic,curves,mfpic}
\newtheorem{theo}{Theorem}[section]
\newtheorem{lemm}[theo]{Lemma}
\newtheorem{conj}[theo]{Conjecture}
\newtheorem{coro}[theo]{Corollary}
\newtheorem{prop}[theo]{Proposition}

\theoremstyle{remark}
\newtheorem{exam}[theo]{Example}
\newtheorem{defi}[theo]{Definition}

\newcommand{\Z}{\mathbb{Z}}

\title{Vertex cuts}
\author{M.J. Dunwoody and B. Kr\"on}

\begin{document}
\maketitle
\begin {dedication}
For Wolfgang Woess on his 60th birthday
\end {dedication}

\begin{abstract}
Given a connected graph, in many cases it is possible to construct a structure tree that provides information about
the ends of the graph or its connectivity.  For example Stallings' theorem on the structure of groups with more than one end
can be proved by analyzing the action of the group on a structure tree and Tutte used a structure tree to investigate finite $2$-connected graphs,
that are not $3$-connected.
Most of these structure tree theories have been  based on edge cuts, which are components of the graph obtained
by removing finitely many edges.
A new axiomatic theory is described here using vertex cuts,  components of the graph obtained by removing finitely many vertices.
This generalizes  Tutte's decomposition of 2-connected graphs to $k$-connected graphs for any $k$, in finite and infinite graphs.
The theory can be applied to non-locally finite graphs with more than one vertex end, i.e.\ ends that can be separated by removing a finite number of vertices.  This gives a decomposition for a group acting on such a graph,  generalizing Stallings' theorem. Further applications include the classification of distance transitive graphs and $k$-CS-transitive graphs.\end{abstract}

\newcommand{\N}{\mathbb{N}}
\newcommand{\ca}{\mathcal{A}}
\newcommand{\cn}{\mathcal{N}}
\newcommand{\cc}{\mathcal{C}}
\newcommand{\cL}{\mathcal{L}}

\newcommand{\cs}{\mathcal{S}}
\newcommand{\ce}{\mathcal{E}}
\newcommand{\cu}{\mathcal{U}}
\newcommand{\ci}{\mathcal{I}}
\newcommand{\cj}{\mathcal{J}}
\newcommand{\ck}{\mathcal{K}}
\newcommand{\cb}{\mathcal{B}}
\newcounter{fig}
\setcounter{fig}{0}

\section{Introduction}

A connected simple graph $X=(VX,EX)$  is said to be $n$-\emph{connected} if
for every pair $u, v$ of distinct vertices there are $n$ paths joining $u$ to $v$ such that every vertex in $VX\setminus\{u,v\}$ 
lies on at most one of the paths.

\begin{exam}\label{exam:1-connect}
If $X$ is not $2$-connected then it has \emph{cut-points}, i.e.\ vertices whose removal disconnects the
graph.   If this happens, then $X$ decomposes into a collection of so-called \emph{blocks}. These are either maximal $2$-connected induced
subgraphs or disconnecting edges (edges such that if they are removed then the graph becomes disconnected). Any two such blocks are either disjoint or they intersect in a cut-point.   Every edge of $X$ lies in exactly one block.
Associated with this decomposition is a tree $T=(VT,ET)$ in which $VT = \cb\cup\cs$, where  $\cs$ is the set of
cut-points, $\cb$ is the set of blocks and there is an edge in $ET$ joining $b \in\cb$ with $s\in\cs$ if and only
if $s\in b$.   If $G$ is a group of automorphisms of $X$ then there is an induced action of $G$ on $T$. Because of this property we call $T$ a \emph{structure tree}.
If $X$ is a finite graph then $T$ will be a finite tree and any action on a finite tree is trivial, i.e.\ there is a vertex or an edge which is fixed by $G$. This is illustrated in Figure~\ref{fig:1block}. 
The number next to a vertex of $T$  indicates the order of the subgroup of the automorphism group fixing that vertex. Note the $2$-colouring of $T$, in which white vertices are cut points and black vertices are blocks.
\end{exam}

\begin{figure}[htbp]
\centering
\begin{tikzpicture}[scale=0.5]
      %   \put(-80,-5)
          {
    \path (0,2) coordinate (p1);
    \path (2,3) coordinate (p2);
    \path (3,2) coordinate (p3);
    \path (1,1) coordinate (p4);
     \draw (p3) -- (1,1);
     \draw (p1) -- (p2) -- (p3) -- cycle;
     \path (intersection of p1--p3 and p2--p4) coordinate (X);
 \filldraw[white] (X) circle (3pt);
\draw (p1) -- (p4) --(p2);
    \filldraw (p1) circle (3pt) ;
     \filldraw (p2) circle (3pt);
  \filldraw (p2) circle (3pt);  
  \filldraw (p3) circle (3pt);  
  \filldraw (p4) circle (3pt);
  \path (6,2) coordinate (q1);
    \path (8,3) coordinate (q2);
    \path (9,2) coordinate (q3);
    \path (7,1) coordinate (q4);
     \draw (q3) -- (q4);
      \draw (q1) -- (q2) -- (q3) -- cycle;
     \path (intersection of q1--q3 and q2--q4) coordinate (Y);
 \filldraw[white] (Y) circle (3pt);
\draw (q1) -- (q4) --(q2);
    \filldraw (q1) circle (3pt) ;
    \filldraw (q2) circle (3pt);
  \filldraw (q2) circle (3pt);  
  \filldraw (q3) circle (3pt);  
  \filldraw (q4) circle (3pt);
  \path (6,2) coordinate (p1);
    \path (2,3) coordinate (p2);
    \path (5,4) coordinate (p3);
    \draw (p1) -- (p2) -- (p3) -- cycle;
    \filldraw (p2) circle (3pt);
  \filldraw (p2) circle (3pt);  
  \filldraw (p3) circle (3pt);
  \path (2, 0) coordinate (p1);
    \path (2, 1) coordinate (p2);
    \path (3,2) coordinate (p3);
    \path (3,1) coordinate (p4);
     \draw (p3) -- (p4);
     \draw (p1) -- (p2) -- (p3) -- cycle;
     \path (intersection of p1--p3 and p2--p4) coordinate (X);
 \filldraw[white] (X) circle (3pt);
\draw (p1) -- (p4) --(p2);
    \filldraw (p1) circle (3pt) ;
    \filldraw (p2) circle (3pt);
  \filldraw (p2) circle (3pt);  
  \filldraw (p3) circle (3pt); 
  \filldraw (p4) circle (3pt);
  \path (4,5) coordinate (p1);
    \path (3, 4) coordinate (p2);
    \path (2, 5) coordinate (p3);
    \path (3, 7) coordinate (p4);
     \draw (p3) -- (p4);
     \draw (p1) -- (p2) -- (p3) -- cycle;
     \path (intersection of p1--p3 and p2--p4) coordinate (X);
 \filldraw[white] (X) circle (3pt);
\draw (p1) -- (p4) --(p2);
    \filldraw (p1) circle (3pt) ;
     \filldraw (p2) circle (3pt);
  \filldraw (p3) circle (3pt);  
  \filldraw (p4) circle (3pt);
  \path (4,5) coordinate (p1);
    \path (6,6) coordinate (p2);
    \path (7,5) coordinate (p3);
    \path (5,4) coordinate (p4);
     \draw (p3) -- (p4);
     \draw (p1) -- (p2) -- (p3) -- cycle;
     \path (intersection of p1--p3 and p2--p4) coordinate (X);
 \filldraw[white] (X) circle (3pt);
\draw (p1) -- (p4) --(p2);
    \filldraw (p1) circle (3pt) ;
    \filldraw (p2) circle (3pt);
  \filldraw (p2) circle (3pt);  
  \filldraw (p3) circle (3pt);  
  \filldraw (p4) circle (3pt);}
  
% \put(50,-5)
 {\path (6,-8) coordinate (p1);
    \path (2,-7) coordinate (p2);
    \path (5,-6) coordinate (p3);
   \draw (4, -7) coordinate (q1);
   \draw (3, -5) coordinate (q3);
    \draw (5.5, -5) coordinate (q4);
    \draw (1.5, -8) coordinate (q5);
    \draw (2.5, -9) coordinate (q6);
     \filldraw (p1) circle (3pt) ;
 \filldraw (p2) circle (3pt);
   \filldraw (p3) circle (3pt);  
 \filldraw (q1) circle (3pt);
\path (4, -5) coordinate (p4);
 \filldraw (p4) circle (3pt);  
 \path (3, -8) coordinate (p5);
  \filldraw (p5) circle (3pt);  
 \draw (p1) -- (q1);
  \draw (p2) -- (q1);
 \draw (p3) -- (q1);
 \path (7.5, -8) coordinate (q2);
 \filldraw (q2) circle (3pt);
\filldraw (q3) circle (3pt);
\filldraw (q4) circle (3pt);
\filldraw (q5) circle (3pt);
\filldraw (q6) circle (3pt);
\draw (q3)--(q4)--(p3);
 \draw (p2)--(q5)--(p5)--(q6);  
 \draw (p1) --(q2);
     \filldraw [white] (p1) circle (2pt);
\filldraw [white]  (p2) circle (2pt);
\filldraw [white] (p3) circle (2pt);
\filldraw [white] (p5) circle (2pt);
\filldraw [white] (p4) circle (2pt);
\draw (p1) node [below] {$_{1728}$ };
\draw (p2) node [left] {$_{864}$ };
\draw (q1) node [right] {$_{\ 1728}$ };
\draw (q2) node [below] {$_{1728}$ };
\draw (q3) node [left] {$_{864}$ };
\draw (p4) node [above] {$_{864}$ };
\draw (p3) node [right] {$_{864}$ };
\draw (q6) node [below] {$_{864}$ };
\draw (q4) node [above] {$_{864}$ };
\draw (q5) node [below] {$_{864}$ };
\draw (p5) node [right] {$_{864}$ };}
\end{tikzpicture}
\caption{One-connected graph and structure tree}\label{fig:1block}
\end{figure}
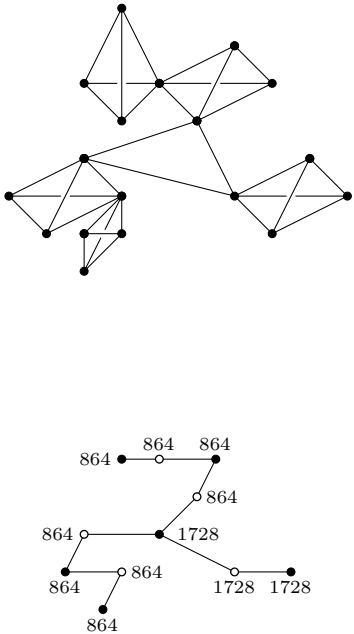

There are  similar decompositions if $X$ is $2$-connected but not $3$-connected.
One  was described by Tutte \cite{Tutte1984} if $X$ is finite.  This was generalised by Droms Servatius and Servatius \cite{Droms2001} to  locally finite graphs and by Richter \cite{Richter2004} to  arbitrary graphs.  
A somewhat different account is given in  \cite{Dunwoody2007}.
The decomposition we will generalise  gives a structure tree $T$, which again has a $2$-coloring $VT =\cb\cup\cs$.

Sets of at least three vertices which never lie in different components after removing any two vertices from the graph are called \emph{2-inseparable}. The set $\cb$ contains two types of blocks. Blocks of type 1 are maximal 2-inseparable sets. Blocks of type 2 correspond to sets of vertices that can be separated by removing two vertices, but cannot be separated by removing two vertices which are contained in some 2-inseparable set. Blocks of type 2 correspond to circular arrangements of blocks of type 1.
The set $\cs$ corresponds to those pairs of vertices whose 
removal disconnects the graph and which are contained in some 2-inseparable set.

\begin{exam}\label{exam:2-connect}
An example is given in Figure~\ref{fig:2block}. There are 4 blocks of type 1 which are arranged in a 4-cycle and one block of type 2, namely  $\{x_1,x_2,x_3,x_4\}$. The 2-separators are
\[\cs=\{\{ x_1, x_2\},\{ x_2, x_3\},\{ x_3, x_4\},\{ x_4, x_1\}\}.\]
Here $\{ x_1, x_3\}$ and $\{ x_2, x_4\}$ are the $2$-separators which are not in $\cs$, because they are not contained in any 2-inseparable set.
The 4-cycle of blocks corresponds
to the $4$-cycle $(x_1,x_2,x_3,x_4)$ of vertices, one edge of which, shown dashed in Figure~\ref{fig:2block}a, is what Tutte calls an \emph{ideal edge}. This edge is not in the original graph, and joins the vertices of a $2$-separator in $\cs$.    Tutte's decomposition for this graph gives a structure tree that is the same as our tree except that the vertices of degree two are removed. Note that we could add  ideal edges so that all blocks become complete graphs as in Figure~\ref{fig:2block}b. The corresponding structure tree is shown Figure~\ref{fig:2block}c.
Tutte's decompositions for 
Figure~\ref{fig:2block}a and Figure~\ref{fig:2block}d are the same, but our decomposition for Figure~\ref{fig:2block}d is trivial as there is only one maximal $2$-inseparable set.

\vskip1mm
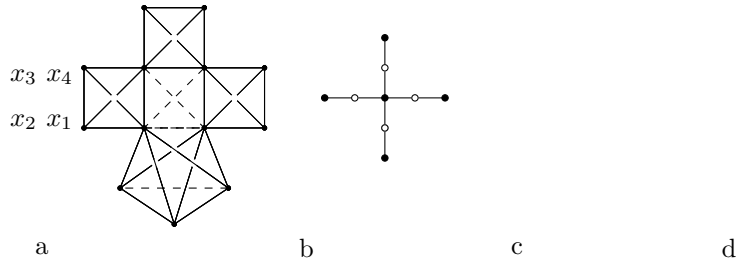
\begin{figure}[htbp]
\centering
\begin{tikzpicture}[scale=0.4]
\put(-50,0){
\path (0,2) coordinate (p1);
    \path (2,2) coordinate (p2);
    \path (0,0) coordinate (p3);
    \path (2,0) coordinate (p4);
       \draw (p1) -- (p2) -- (p3) -- cycle;
     \path (intersection of p2--p3 and p1--p4) coordinate (X);
 \filldraw[white] (X) circle (3pt);
\draw (p2) -- (p4) --(p3);
    \filldraw (p1) circle (2pt) ;
     \filldraw (p2) circle (2pt) ; 
 \draw (2.4,1.7) node  {$x_3$};
     \filldraw (p3) circle (2pt) ;
       \draw (p1) -- (p4);
       \path (2,0) coordinate (p1);
    \path (4,0) coordinate (p2);
    \path (3,-3.2) coordinate (p3);
    \path (4.8, -2) coordinate (p4);
    \path (1.2, -2) coordinate (p5);
    \draw (p3)--(p5)--(p2);
    \draw [dashed] (p1)-- (p2);
\draw (p1) -- (p5) ;
    \draw (p4) -- (p2) -- (p3);
     \path (intersection of p1--p3 and p2--p5) coordinate (Y);
      \path (intersection of p1--p4 and p2--p5) coordinate (Z);
 \path (intersection of p2--p3 and p1--p4) coordinate (X);

 \filldraw[white] (X) circle (3pt);
 \filldraw[white] (Y) circle (3pt);
 \filldraw[white] (Z) circle (3pt);
\draw  (p1) --(p3)--(p4)--cycle ;
    \filldraw (p1) circle (2pt) ;
     \filldraw (p2) circle (2pt) ; 
 \draw (2.4, .2) node {$x_2$};
  \draw (3.6,.2) node {$x_1$};
     \filldraw (p3) circle (2pt) ;
      \filldraw (p4) circle (2pt) ;
 \filldraw (p5) circle (2pt) ;
      \path (6,2) coordinate (p1);
    \path (4,2) coordinate (p2);
    \path (4,0) coordinate (p3);
    \path (6,0) coordinate (p4);
     \draw (p3) -- (p2);
  \filldraw (p4) circle (2pt) ;

   \draw (p1) -- (p2) -- (p4) -- cycle;
     \path (5, 1) coordinate (X);
 \filldraw[white] (X) circle (3pt);
\draw (p4) -- (p3) --(p1);
    \filldraw (p1) circle (2pt) ;
    
      \path (4,2) coordinate (p1);
    \path (2,2) coordinate (p2);
    \path (2, 4) coordinate (p3);
    \path (4,4) coordinate (p4);
     \draw (p4) -- (p1);
     \draw (p1) -- (p2) -- (p3) -- cycle;
     \path (3,3) coordinate (X);
 \filldraw[white] (X) circle (3pt);
\draw (p2) -- (p4) --(p3);
    \filldraw (p1) circle (2pt) ;
 \draw (3.6, 1.7) node {$x_4$};
    \filldraw (p4) circle (2pt) ;
  \filldraw (p3) circle (2pt) ;
\draw (3,-4) node {a};}

\put(50,0){
    \path (0,2) coordinate (p1);
    \path (2,2) coordinate (p2);
    \path (0,0) coordinate (p3);
    \path (2,0) coordinate (p4);
       \draw (p1) -- (p2) -- (p3) -- cycle;
     \path (intersection of p2--p3 and p1--p4) coordinate (X);
 \filldraw[white] (X) circle (3pt);
\draw (p2) -- (p4) --(p3);
    \filldraw (p1) circle (2pt) ;
     \filldraw (p2) circle (2pt) ; 
     \filldraw (p3) circle (2pt) ;
       \draw (p1) -- (p4);

    \path (2,0) coordinate (p1);
    \path (4,0) coordinate (p2);
    \path (3,-3.2) coordinate (p3);
    \path (4.8, -2) coordinate (p4);
    \path (1.2, -2) coordinate (p5);
    \draw (p3)--(p5)--(p2);
    \draw [dashed] (p4)-- (p5);
\draw (p1) -- (p5) ;
    \draw (p4) -- (p2) -- (p3);
     \path (intersection of p1--p3 and p2--p5) coordinate (Y);
      \path (intersection of p1--p4 and p2--p5) coordinate (Z);
 \path (intersection of p2--p3 and p1--p4) coordinate (X);

 \filldraw[white] (X) circle (3pt);
 \filldraw[white] (Y) circle (3pt);
 \filldraw[white] (Z) circle (3pt);
\draw  (p1) --(p3)--(p4)--cycle ;
    \filldraw (p1) circle (2pt) ;
     \filldraw (p2) circle (2pt) ; 
     \filldraw (p3) circle (2pt) ;
      \filldraw (p4) circle (2pt) ;
 \filldraw (p5) circle (2pt) ;
      \path (6,2) coordinate (r1);
    \path (4,2) coordinate (r2);
    \path (4,0) coordinate (r3);
    \path (6,0) coordinate (r4);
     \draw (r3) -- (r2);
  \filldraw (r4) circle (2pt) ;

    \draw (r1) -- (r2) -- (r4) -- cycle;
     \path (5, 1) coordinate (X);
 \filldraw[white] (X) circle (3pt);
\draw (r4) -- (r3) --(r1);
    \filldraw (r1) circle (2pt) ;
    
      \path (4,2) coordinate (s1);
    \path (2,2) coordinate (s2);
    \path (2, 4) coordinate (s3);
    \path (4,4) coordinate (s4);
     \draw (s4) -- (s1);
     \draw (s1) -- (s2) -- (s3) -- cycle;
     \path (3,3) coordinate (X);
 \filldraw[white] (X) circle (3pt);
\draw (s2) -- (s4) --(s3);
    \filldraw (s1) circle (2pt) ;
    \filldraw (s4) circle (2pt) ;
  \filldraw (s3) circle (2pt) ;
\draw [dashed] (s2)--(p2)-- (p1)--(s1);
\draw (3,-4) node {b};
\path (8,1) coordinate (p1);
    \path (12,1) coordinate (p2);
    \path (10, -1) coordinate (p3);
   \draw (10, 3) coordinate (q1);
   \draw (10,1) coordinate (q3);
    \draw (9,1) coordinate (q4);
    \draw (11,1) coordinate (q5);
    \draw (10,0) coordinate (q6);
    \draw (10, 2) coordinate (q7);
  
     \filldraw (p1) circle (3pt) ;
 \filldraw (p2) circle (3pt);
   \filldraw (p3) circle (3pt);  
 \filldraw (q1) circle (3pt);
\filldraw (q6) circle (3pt) ;

 \draw (p1) -- (p2);
  \draw (p3) -- (q1);
 
\filldraw (q3) circle (3pt);
\filldraw [white] (q4) circle (3pt);
\filldraw [white] (q5) circle (3pt);
\filldraw [white] (q6) circle (3pt);
\filldraw [white] (q7) circle (3pt);
\draw (q4) circle (3pt);
\draw (q5) circle (3pt);
\draw (q6) circle (3pt);
\draw (q7) circle (3pt);  

\draw (10,-4) node {c};}
\put(210,0){
\path (0,2) coordinate (p1);
    \path (2,2) coordinate (p2);
    \path (0,0) coordinate (p3);
    \path (2,0) coordinate (p4);
     \draw (p1) -- (p3) -- (p4) --(p2) --cycle;
     \path (intersection of p2--p3 and p1--p4) coordinate (X);
 \filldraw[white] (X) circle (3pt);
%\draw (p2) -- (p4) --(p3);
    \filldraw (p1) circle (2pt) ;
     \filldraw (p2) circle (2pt) ; 
 %\draw (2.4,1.7) node  {$x_3$};
     \filldraw (p3) circle (2pt) ;
       %\draw (p1) -- (p4);
       \path (2,0) coordinate (p1);
    \path (4,0) coordinate (p2);
    \path (3,-3.2) coordinate (p3);
    \path (4.8, -2) coordinate (p4);
    \path (1.2, -2) coordinate (p5);
   \draw (p3)--(p5)--(p2);
    \draw [dashed] (p1)-- (p2);
\draw (p1) -- (p5) ;
    \draw (p4) -- (p2) -- (p3);
     \path (intersection of p1--p3 and p2--p5) coordinate (Y);
      \path (intersection of p1--p4 and p2--p5) coordinate (Z);
 \path (intersection of p2--p3 and p1--p4) coordinate (X);

 \filldraw[white] (X) circle (3pt);
 \filldraw[white] (Y) circle (3pt);
 \filldraw[white] (Z) circle (3pt);
\draw  (p1) --(p3)--(p4)--cycle ;
    \filldraw (p1) circle (2pt) ;
     \filldraw (p2) circle (2pt) ; 
% \draw (2.4, .2) node {$x_2$};
%  \draw (3.6,.2) node {$x_1$};
     \filldraw (p3) circle (2pt) ;
      \filldraw (p4) circle (2pt) ;
 \filldraw (p5) circle (2pt) ;
      \path (6,2) coordinate (p1);
    \path (4,2) coordinate (p2);
    \path (4,0) coordinate (p3);
    \path (6,0) coordinate (p4);
     \draw (p3) -- (p2);
  \filldraw (p4) circle (2pt) ;

   \draw  (p4) --(p1)-- (p2) ;
     \path (5, 1) coordinate (X);
 \filldraw[white] (X) circle (3pt);
\draw (p4) -- (p3) ;
    \filldraw (p1) circle (2pt) ;
    
      \path (4,2) coordinate (p1);
    \path (2,2) coordinate (p2);
    \path (2, 4) coordinate (p3);
    \path (4,4) coordinate (p4);
     \draw (p4) -- (p1)--(p2) ;
     \draw  (p2) -- (p3) ;
     \path (3,3) coordinate (X);
 \filldraw[white] (X) circle (3pt);
\draw (p4) --(p3);
    \filldraw (p1) circle (2pt) ;
% \draw (3.6, 1.7) node {$x_4$};
    \filldraw (p4) circle (2pt) ;
  \filldraw (p3) circle (2pt) ;
\draw (3,-4) node {d};}
\draw (20,0) node {$\ $};
\end{tikzpicture}
\vskip-0.1cm \caption{Decomposition of a 2-connected graph}\label{fig:2block}\vskip-6mm
\end{figure}
\end{exam}

In this paper we develop an axiomatic theory which shows that if  a graph $X$ has a finite set of vertices whose removal produces at least
two components that are large in some sense and $G$ is the automorphism group of $X$, then there is a
$G$-tree  (or structure tree) $T$ with a bipartition $(\cs,\cb)$ of the set of vertices $VT=\cs\cup\cb$ so that the vertices in $\cs$ correspond to
finite separating sets.

Such a structure tree had been known to exist in the case when $X$ is an infinite graph
and $X$ can be disconnected into two infinite components by removing finitely many edges (see \cite{Dicks1989,Dunwoody1982}). We generalize this to the case where infinite components are separated by removing finite sets of vertices. See Examples~\ref{exam:vertexends}, \ref{exam:edgeends} and \ref{exam:arbitraryends}.

We have seen structure trees for finite graphs in Example~\ref{exam:1-connect} for 1-connected graphs which are not 2-connected and in Example~\ref{exam:2-connect} for 2-connected graphs which are not 3-connected.
Within the axiomatic theory we will also generalize Tutte's decomposition to $k$-connected graphs for any $k$, see Example~\ref{exam:finitecuts}.\medskip

 A \emph{ray} is a sequence of distinct vertices $v_0, v_1, \dots $ such that $v_i$ and $v_{i+1}$ are adjacent for each $i$.
Let $C \subset VX$.  The \emph{coboundary } $\delta C$ is the set of edges that have one vertex in
$C$ and one in $VX\setminus C$.  If $\delta C$ is finite, then $C$ is called an \emph{edge  cut}.
If $R$ is a ray, then all the terms from a certain point onwards will either lie in $C$ or in $VX\setminus C$.
We say that $C$ \emph{separates} rays $R_1$, $R_2$ if one of the rays eventually belongs to $C$ and the
other eventually belongs to $VX\setminus C$.
We say two rays  belong to the same edge end if  they are not separated by any edge cut.
It is easy to see that this is an equivalence relation on the set of rays, and so we can
take an equivalence class to be an \emph{edge end}.

In \cite{Dunwoody1982}  it is shown that if a graph has more than one edge end and $k$ is the smallest integer for which
there is an edge cut $C$ with $|\delta C| = k$ that separates two ends, then there is a structure tree
in which the edges correspond to edge cuts with this property.

In \cite{Dicks1989} it is shown that there is a sequence of structure trees $T_n$ invariant under the action of $G$, the automorphism group of 
$X$, with the property that if two edge ends of $X$ are separated by removing $n$ edges then they are separated by a cut $C$ in $ET_n$
such that $|\delta C| \leq n$.    Recently the first author \cite {D13} has shown that there is a such a sequence of structure trees  for any graph - finite or infinite -  that is uniquely determined and which separates distinct vertices or  distinct edge ends or a vertex and an edge end.

A set of vertices $C\subset VX$ is said to be \emph{connected} if any two vertices in $C$ can be joined by a path all of whose vertices are in $C$. \emph{Components} of a set of vertices are its maximal connected subsets.

In this paper we are concerned with  vertex cuts and vertex ends.
We say that $C\subset VX$ is a vertex cut if  $C$ is connected and  $VX$ can be partitioned 
$C\cup NC \cup C^*$, where $NC$ is finite and consists of the vertices which are not in $C$, but which 
are adjacent to vertices in $C$.  Note that generally $C^*$ will not be connected.
As for edge cuts any ray is eventually in $C$ or in $C^*$.   We say two rays  belong to the
same \emph{vertex end} if they are not separated by any vertex cut. A finite set $F$ of vertices is called a \emph{separator} if $VX \setminus F$ has at least two components which contain an end.
If $C$ is an edge cut then it is also a vertex cut in which $NC$ is the set of vertices of $\delta C$ which
are not in $C$.  Thus if two rays belong to the same vertex end, then they belong to the same edge
end.   The converse is true if $X$ is locally finite.   However it is easy to construct examples of graphs which are not locally finite
in which there are more vertex ends than edge ends.  For example if $K_\infty$ is the complete graph
on a countably infinite set of vertices and $X$ is the graph consisting of $n$ copies of $K_\infty$, 
in which  a single vertex from each copy is identified, then $X$ has $n$ vertex ends but only one edge end.

\begin{exam}\label{exam:farey}
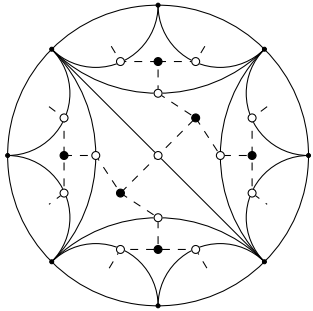
\begin{figure}[htbp]
\centering
\begin{tikzpicture}[scale=0.5]
\draw (2,0) circle (4cm);
    \draw (-.828, -2.828)  arc (-45: 45: 4);
  \draw (-2,0) arc (90: -45: 1.656);
\draw (-2,0) arc (-90: 45: 1.656);
\draw (6,0) arc (-90: 45: -1.656);
\draw (6,0) arc (90: -45: -1.656);
    \draw (4.828, -2.828)  arc (45: 135: 4);
\draw (4.828, -2.828)  arc (45: -45: -4);
\draw (-.828, 2.828)  arc (-135: -45: 4);
\draw (2,4)  arc (0: 135: -1.656);
\draw (2,4)  arc (180: 45: -1.656);
\draw (2,-4)  arc (0: -135: -1.656);
\draw (2,-4)  arc (-180: -45: -1.656);
\draw (4.828, -2.828) -- (-.828, 2.828);
  \path (2,0) coordinate (p1);
    \path (3,1) coordinate (p2);
    \path (1, -1) coordinate (p3);
       \filldraw (p2) circle (3pt);
   \filldraw (p3) circle (3pt);
     \path (2,2.5) coordinate (q1);
    \path (2,-2.5) coordinate (q2);
    \path (-.5,0) coordinate (q3);
      \path (4.5,0) coordinate (q4);
   \filldraw (q1) circle (3pt);
   \filldraw (q2) circle (3pt);
   \filldraw (q3) circle (3pt);
   \filldraw (q4) circle (3pt);
      \path (.344,0) coordinate (r1);
    \path (3.656, 0) coordinate (r2);
    \path (2,  1.656) coordinate (r3);
      \path (2, -1.656) coordinate (r4);
    \path (-.5,1) coordinate (u1);
    \path (-.5, -1) coordinate (u2);
    \path (4.5, 1) coordinate (u3);
      \path (4.5, -1) coordinate (u4);
 \path (1,2.5) coordinate (v1);
    \path (3,2.5) coordinate (v2);
    \path (1,-2.5) coordinate (v3);
      \path (3,-2.5) coordinate (v4);
      \draw [dashed]  (r3)--(p2) -- (p3) --(r1);
\draw [dashed]  (p3)--(r4) -- (q2);
\draw [dashed]  (p2)--(r2) -- (q4);
\draw [dashed]  (r3)--(q1) -- (v2)-- (3.3, 3);
\draw [dashed]  (q1) -- (v1) -- (.7, 3);
\draw [dashed]  (-.9, 1.3)--(u1)--(u2) -- (-.9, -1.3);
\draw [dashed]  (4.9, 1.3) -- (u3)--(u4)-- (4.9, -1.3);
\draw [dashed]  (q3)--(r1);
\draw [dashed]  (.7, -3)--(v3)--(v4)-- (3.3, -3);
   \filldraw [white] (v1) circle (3pt);
   \filldraw [white] (v2) circle (3pt);
   \filldraw [white] (v3) circle (3pt);
   \filldraw [white] (v4) circle (3pt);
   \filldraw [white] (u1) circle (3pt);
   \filldraw [white] (u2) circle (3pt);
   \filldraw [white] (u3) circle (3pt);
   \filldraw [white] (u4) circle (3pt);
   \filldraw [white] (r1) circle (3pt);
   \filldraw [white] (r2) circle (3pt);
   \filldraw [white] (r3) circle (3pt);
   \filldraw [white] (r4) circle (3pt);
    \filldraw [white] (p1) circle (3pt);
    
   \draw (p1) circle (3pt);
   \draw (r1) circle (3pt);
   \draw (r2) circle (3pt);
   \draw (r3) circle (3pt);
   \draw (r4) circle (3pt);
   \draw (v1) circle (3pt);
   \draw (v2) circle (3pt);
   \draw (v3) circle (3pt);
   \draw (v4) circle (3pt);
   \draw (u1) circle (3pt);
   \draw (u2) circle (3pt);
   \draw (u3) circle (3pt);
   \draw (u4) circle (3pt);
    
  \path (-2,0) coordinate (p1);
    \path (6,0) coordinate (p2);
    \path (2, 4) coordinate (p3);
     \path (2, -4) coordinate (p4);
        \filldraw (p1) circle (1.5pt);
   \filldraw (p2) circle (1.5pt);
   \filldraw (p3) circle (1.5pt);
   \filldraw (p4) circle (1.5pt);
   \path (4.828,-2.828) coordinate (p1);
    \path (4.828,2.828) coordinate (p2);
    \path (-.828, -2.828) coordinate (p3);
     \path (-.828, 2.828) coordinate (p4);
     \filldraw (p1) circle (1.5pt);
   \filldraw (p2) circle (1.5pt);
   \filldraw (p3) circle (1.5pt);
   \filldraw (p4) circle (1.5pt);

\end{tikzpicture}
\vskip-1mm\caption{Farey graph with structure tree.}\label{fig:Farey}\vskip-2mm
\end{figure}
For the Farey graph vertex cuts yield a tree decomposition but edge cuts do not, see Figure~\ref{fig:Farey}. This example was pointed out to us by Hamish Short.
This graph is obtained by taking an ideal triangle in the hyperbolic plane and then taking all translates of this triangle under the group
of isometries generated by reflexions in the three sides.  One obtains a graph, in which the vertices are the translates of the vertices
of the triangle.  All of these will lie in the boundary of the plane, which will be a circle in the disc model.  The edges of the graph will be the translates of the edges of the triangle.   The vertices of any edge will form a $2$-separator.    In this graph every vertex has infinite valency.
The structure tree is easy to see.   There will be one orbit of vertices corresponding to the separating edges. The other orbit
corresponds to the triangles. For each such triangle there will be three  edges of the structure tree joining the vertex corresponding to the triangle
to the three vertices corresponding to its boundary edges.    The structure tree is essentially the dual graph to the tessellation of the hyperbolic plane.
This dual graph is a tree since each edge of the original graph is separating.
\end{exam}

In developing our theory we give a set of axioms that it is sufficient for a  set of vertex cuts to satisfy in order that a structure tree can be constructed.
In Figure~\ref{fig:2block} removing any two of the central four vertices will leave two components.   The $12$ components thus obtained satisfy
the axioms of a \emph{cut system}.  The $12$ cuts are not \emph {nested}  with each other.  Thus if one takes two cuts $C, D$ such that $NC, ND$ are the two distinct diagonal pairs, then $C\cap D, C \cap D^*, C^*\cap D, C^*\cap D^*$ are  all non-empty.  But if we restrict the cut system to the components obtained by removing a separator in 
$\cs=\{\{ x_1, x_2\},\{ x_2, x_3\},\{ x_3, x_4\},\{ x_4, x_1\}\}$ (later these will be referred to as ``optimally nested cuts'') then we obtain a cut system in which all pairs of cuts are nested. The cuts in such a system can be regarded as the directed edges of the structure tree.

If $X$ is an infinite graph with more than one vertex end and $k$ is the smallest integer
for which there is a vertex cut $C$ such that $|NC| = k$ and $NC$ separates two ends
then there is a set of such vertex cuts which satisfies the axioms.

We  obtain a structure tree theory
that applies to  finite graphs, and gives information about the $k$-connectivity of the graph for any $k$.
 For a complete graph the structure tree is trivial, i.e.\ it only has one vertex.
 %Any  connected graph $X$  has a  non-trivial structure tree, withif and only if for some integer $k$, there are $k$-separators.
 A $k$-separator is  a set $S$ of $k$ vertices
 whose removal  leaves at least two components $C,D$ such that $C\cup S$ and $D\cup S$ each contain
  $k$-inseparable subsets.   Here a set $Y$ of vertices is $k$-inseparable if $Y$ has at least $k+1$ vertices and no two of the vertices 
 will lie in distinct components of the graph when at most $k$ vertices of $X$ are removed.  
 Let $\kappa $ be the smallest value of $k$ for which the above occurs.  We show that any connected  graph  contains a uniquely  determined nested set  $\cn$  of $\kappa $-cuts  such that if two $\kappa $-inseparable subsets are separated
 by some $\kappa $-separator then they are separated by a set in $\cn $. The uniqueness of the  set $\cn$ means that it is  invariant under the automorphism group of $X$ and
 forms the directed edge set of a structure tree $T_{\kappa }$  for $X$.    We also show that $2$-colouring $T_{\kappa }$ partitions $VT_{\kappa }$ into
 {\it blocks} and  separators.   %Each separator contains $\kappa $ vertices and removing these vertices separates the graph.
 Every $k$-inseparable set is a subset of a block,  though some blocks may not contain any $k$-inseparable set.
 
 Unlike the situation for edge cuts, it is not usually possible to construct  a uniquely determined sequence
 of trees corresponding to vertex cuts in $X$ which separate any pair of vertex ends or pairs of distinct maximal $k$-inseparable sets.
 However we show that for any pair of distinct vertex ends 
 there is a uniquely determined sequence  of structure trees starting with $T_{\kappa }$ and such that each succeeding tree is a structure tree in a graph whose vertex set is  a single block of the preceding tree until one obtains a structure tree in which the pair of ends  do not lie in the same block,
 and so are separated by a separator of that tree.     In a similar way one can construct a sequence of structure trees to separate a pair of maximal
 $k$-inseparable sets.   However in this case it may happen - in a rather exceptional case that we specify - that the the pair of sets end up in a block in which they cannot be separated.
 \subsection {Index of Definitions}
 \ \ 
 \medskip
  \halign { #    &  #  &  #   & # \cr

\ref {cuts} &  cuts.& \ref {nested} & nested\hskip 2pt cuts.\cr
 \ref{pre-cut} & pre-cut.&  \ref{slice} & slice. \cr
 \ref{separator} & separator.&  \ref{optimal} & optimally\hskip 2pt nested.\cr
 \ref{sep AB}  &$C$ separates $A$ and $B$. \ \ \ \  & \ref{block} & block. \cr
 \ref{insep} & $k$-inseparable.& \ref{omega} & $\Omega $-cut\hskip 2pt system. \cr
 \ref{thin} &  thin\hskip 2pt cut.  & \ref{opti-wrt} & optimally nested with respect to $\omega _1$ and $\omega _2 $.\cr
 \ref {isolated}&  isolated\hskip 2pt corner. &\  & \cr
 }
 
 \section{Systems of cuts and separators}

The \emph{boundary} $NC$ of a set of vertices $C$ is the set of vertices in $VX\setminus C$ which are adjacent to $C$. 
Set $C^*=VX\setminus (C\cup NC)$. We call $C^*$ the \emph{$*$-complement} of $C$. 

Let $C$ and $D$ be sets of vertices. The intersections $C\cap D$, $C^*\cap D$, $C\cap D^*$ and $C^*\cap D^*$ are called the \emph{corners} of $C$ and $D$, see Figure~\ref{fig:cross}. The sets $C\cap ND$, $C^*\cap ND$, $D\cap NC$ and $D^*\cap NC$ are called the \emph{links} and $NC\cap ND$ is the \emph{centre}. A link and a corner are said to be \emph{adjacent} if they are adjacent in Figure~\ref  {fig:cross}. We say that two links are the links \emph{of} their adjacent corner (or we say they are \emph{its} links), and we say that two corners are the corners \emph{of} their adjacent link (or \emph{its} corners). Two links or two corners are said to be \emph{adjacent} if they are adjacent to the same link or corner, respectively. Otherwise they are called \emph{opposite}.

\begin {defi} \label {cuts} Let $\cc$ be a collection  of non-empty connected sets of vertices with finite
boundaries in a connected graph. We call $\cc$ a \emph{cut system} if $(C^*)^*=C$ for all $C\in\cc$ and any
pair $C$, $D$ of elements of $\cc$ satisfy the following.

\begin{itemize}
\item[(A1)] If two opposite corners of $C$ and $D$ each contain an element of $\cc$ then any of their
components which contains an element of $\cc$ is in $\cc$.
\item[(A2)] There are two opposite corners of $C$ and $D$ each of which contains an element of $\cc$.
\end{itemize}
\end {defi}
Elements of a cut-system are called \emph{cuts}. \emph{Cut components} are components which are cuts.

\begin{lemm}
If $C$ is a cut then $C^*$ has a cut component and every component of $C^*$ which contains a cut is a cut.
\end{lemm}

\begin{proof}
For $C=D$ Axiom (A2) implies that $C^*$ contains a cut and (A1) implies that every component of $C^*$ which contains a cut is itself a cut.
\end{proof}

Let $C'$ denote the set of vertices in $NC$ which are not adjacent to any element of $C^*$.

\begin{lemm}\label{lemm:invol}
Let $C,D$ be sets of vertices. Then $C\subset (C^*)^*$ and $C^*=((C^*)^*)^*$.

Suppose $C=(C^*)^*$ and $D= (D^*)^*$ and let $A$ be a corner of $C,D$. Then $A=(A^*)^*$. The latter also holds for any component of $A$.
\end{lemm}

\begin{proof}
We have $(C^*)^*=C\cup C'$ and thus $((C^*)^*)^*=(C\cup C')^*=C^*$.

Let us assume that $C=(C^*)^*$ and $D= (D^*)^*$ and let $A$ be a corner of $C,D$. If say $A=C\cap D$ then any vertex $x$ in $NA$ has a  neighbour in $C^*\cup D^*$, because if $x\in NC$ then $C'=\emptyset$ implies that $x$ has a  neighbour in $C^*$ and if $x\in ND$ then $D'=\emptyset$ implies that $x$ has a  neighbour in $D^*$. Thus $A'=\emptyset$ and the same holds for any component of $A$. The other three cases for $A$ being a corner of $C,D$ are treated similarly.
\end{proof}

\begin {defi} \label {pre-cut} If $\cc$ is a cut-system then elements of $\cc\cup\cc^*$ are called \emph{pre-cuts}.
\end {defi} 
 Note that the $*$-complement is a bijective involution on $\cc\cup\cc^*$.

\begin{lemm}\label{lemm:compl_incl}
Let $C$ and $D$ be sets of vertices. Then $C\subset D$ implies $D^*\subset C^*$. If $(D^*)^*=D$ then $D^*\subset C^*$ implies $C\subset D$.
If $C,D$ are pre-cuts then $C\subset D$ if and only if $D^*\subset C^*$.
\end{lemm}

\begin{proof}
The first implication is clear from the definition of the $*$-complement. If $(D^*)^*=D$ and $D^*\subset C^*$ then the first implication and Lemma~\ref{lemm:invol} imply 
$C\subset (C^*)^*\subset (D^*)^*=D$. The statement for pre-cuts now follows from the fact that $(C^*)^*=C$ for any pre-cut $C$.
\end{proof}

The condition $(D^*)^*=D$ in Lemma~\ref{lemm:compl_incl} is necessary as the following example shows.

\begin{exam}
Consider the graph $X=(\{x,y\},\{\{x,y\}\})$ with just one edge. Set $D=\{x\}$ and $C=\{x,y\}$. Then $C^*=D^*=\emptyset$ and $(D^*)^*\ne D$. We have $D^*\subset C^*$ but $C\not\subset D$.
\end{exam}

\begin {defi}\label {separator} If $\cc$ is a cut system then the boundary of a cut is called a \emph{separator} \end {defi}

 We denote the set of separators by $\cs$.
Note that in general the set of separators does not determine a cut system.

Axiom (A2) can equivalently be replaced by the following.

\begin{itemize}
\item[(A2$'$)] If $C$ and $D$ are in $\mathcal{C}$ then both $C\setminus ND$ and $C^*\setminus ND$ contain an element of $\cc$.
\end{itemize}

To see the implication (A2$'$) $\Rightarrow$ (A2) note that one has to apply (A2$'$) also with $C$ and $D$ swapped.
Let us consider some examples of cut systems before developing the theory.

\begin{exam}\label{exam:vertexends}
For infinite graphs with more than one (vertex) end, take $\cc$ to be the set connected sets of vertices $C$ with $(C^*)^*=C$, $|NC|<\infty$, and such that $C$ separates two rays (i.e.\ two  ends). 
\end{exam}

\begin{proof}
If $C$ and $D$ are in $\cc$, then any component of a corner of $C$ and $D$ which contains a ray is also in $\cc$. Here we use the second part of Lemma~\ref{lemm:invol}.
This implies Axiom~(A1).
Axiom~(A2) holds because if $C$, $D$, $C^*$, $D^*$  all contain rays, then so do two opposite corners.
\end{proof}

\begin{exam}\label{exam:edgeends}
For infinite graphs with more than one edge end, the separators would naturally be finite sets of edges which separate rays. But separators are by definition sets of vertices. Hence we replace every edge of the original graph by paths of length two. Let $M$ be the set of new vertices, that is, the set of middle vertices of these paths of length two. Then we take cuts to be connected sets $C$ with $(C^*)^*=C$, such that $C$ separate two rays (i.e.\ ends) and their boundary $NC$ is a finite subset of $M$.
\end{exam}

\begin {defi} \label {sep AB}  We say that a cut $C$ \emph{separates} a set $A$ from a set $B$ if $A\subset C\cup NC$ and $B\subset C^*\cup NC$ or $B\subset C\cup NC$ and $A\subset C^*\cup NC$. Here we also require that neither $A$ nor $B$ is a subset of $NC$.   A  separator $S$ is said to \emph{separate} $A$ from $B$ if for some cut $C$ with $NC=S$, $C$ separates $A$ and $B$.
\end {defi} 
\begin{exam}\label{exam:arbitraryends}
Examples~\ref{exam:vertexends} and \ref{exam:edgeends} can be generalized by choosing $M$ as an arbitrary set of vertices.
\end{exam}

Next we consider a cut system which also makes sense in finite graphs.
\begin {defi} \label {insep} Let $k$ be a positive integer.
A subset $Y$ of $VX$ is said to be \emph{$k$-inseparable} if  it has at least $k+1$ elements and if for every set $C\subset VX$ with
$|NC| \le k$, either  $Y \subset C\cup NC$ or $Y \subset C^*\cup NC$.
\end {defi}
Examples of $k$-inseparable subgraphs are the vertex set of a  $(k+1)$-connected subgraph, or
the vertex set of a subgraph which is complete on $k+1$ vertices.
The vertices of a separating edge form a maximal $1$-inseparable set.

\begin{prop}
Let $X$ be a graph and let $Y \subset VX$ be a $k$-inseparable set.   There is a unique maximal $k$-inseparable set  containing $Y$.
\end {prop}
\begin {proof}
We show that if $A, B$ are $k$-inseparable sets containing $Y$ then $A\cup B$ is $k$-inseparable.  For suppose that $C \subset VX$ and
$|NC| \leq k$, then we know that either $Y \subset C\cup NC$ or $Y\subset C^*\cup NC$.   By relabelling $C$ as $C^*$ if necessary, we can assume
that $Y \subset C\cup NC$.   But now also $A \subset C\cup NC$ and $B\subset C\cup NC$, since $A, B$ are $k$-inseparable.
Thus $A\cup B$ is $k$-inseparable.
An ascending union of $k$-inseparable sets is $k$-inseparable and the proposition follows immediately.
\end {proof}

\begin{exam}\label{exam:finitecuts}
Let $\kappa$ be the smallest positive integer for which there are sets $C$, $Y_1$ and $Y_2$ such that $|NC|=\kappa$, $Y_1$ and $Y_2$ are $\kappa$-inseparable, $Y_1 \subset  C\cup NC$
and $Y_2 \subset C^* \cup NC$. We will show shortly that such connected sets $C$ with this property form a cut system. Examples for $\kappa=1,2$ are the graphs in Examples~\ref{exam:1-connect} and \ref{exam:2-connect}.
\end{exam}

Note that there are graphs where this cut system is empty. It is easy to see that this is the case for finite complete graphs or cycles. In fact, this holds for all finite transitive graphs, see Corollary~\ref{coro:finitetrans}.

Set $a=|C\cap ND|$, $b=|D^*\cap NC|$, $c=|C^*\cap ND|$, $d=|D\cap NC|$ and $m=|NC\cap ND|$, see Figure~\ref{fig:cross}.

\begin{proof}[Proof for Example~\ref{exam:finitecuts}]
Let $C,D$ be  cuts. We want to show that there exist opposite corners which have components that are cuts.
We know that there are $\kappa$-inseparable sets $Y_1$, $Y_2$ separated by $NC$,
and $\kappa $-inseparable sets $Y_3$, $Y_4$ separated by $ND$.
Each $Y_i$ determines a unique corner $A_i$ of $C$ and $D$ such that $Y_i$ is contained in the union of $A_i$, the two links which are adjacent to $A_i$ and the center.  Even though $A_i$ is uniquely determined, we cannot rule out the possibility
that $Y_i \cap A_i=\emptyset$ at this stage.
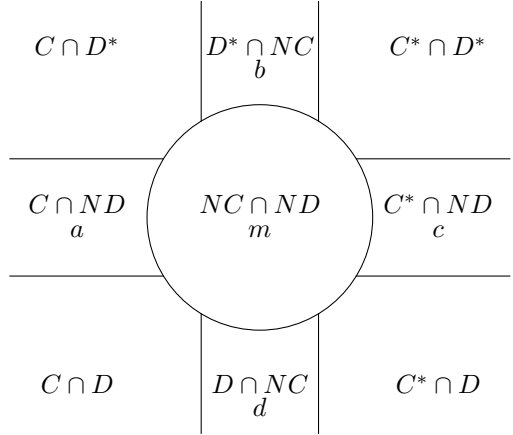
\begin{figure}[htbp]
\begin{tikzpicture}[scale=0.03]

\draw(85,46.5)--(85,100);
\draw(137,46.5)-- (137,100);
\draw(85,-38.5)--(85,-92);
\draw(137,-38.5)--(137,-92);
\draw(68.5,30)--(0,30);
\draw(68.5,-22)--(0,-22);
\draw(222,30)--(153.5,30);
\draw(222,-22)--(153.5,-22);
\draw (111,4) circle(50);
\draw (30, 80) node {$C\cap D^*$} ;
\draw (111, 80) node {$D^*\cap NC$} ;
\draw (190, 80) node {$C^*\cap D^*$} ;
\draw (30, 10) node {$C\cap ND$} ;
\draw (30, -70) node {$C\cap D$} ;
\draw (111, -70) node {$D\cap NC$} ;
\draw (111,10) node {$NC\cap ND$} ;
\draw (190, 10) node {$C^*\cap ND$} ;
\draw (190, -70) node {$C^*\cap D$} ;
\draw (30, -2) node {$a$} ;
\draw (111, 70) node {$b$} ;
\draw (190, -2) node {$c$} ;
\draw (111, -80) node {$d$} ;
\draw (111, -2) node {$m$} ;
\end{tikzpicture}
%\vskip-2cm
\caption{Corners, links and centre}\label{fig:cross}\vskip-3mm
\end{figure}

Two different $Y_i$ may determine the same corner.  There must be two $Y_i$'s which determine opposite corners.
For if this is not the case then all four $Y_i$'s will determine one of a pair of adjacent corners.   But then there will either  be no $Y_i$'s separated by
$NC$ or no $Y_i$'s separated by $ND$.
Suppose then, say, that
\[Y_1 \subset (C\cap D)\cup (C\cap ND)\cup (D\cap NC)\cup (NC\cap ND)\mbox{\quad and}\]
\[Y_2 \subset (C^*\cap D^*)\cup (C^*\cap ND)\cup (D^*\cap NC)\cup (NC\cap ND).\]

Consider Figure~\ref{fig:cross}.   We see that $|N(C\cap D)| \leq  a + m + d $ and $|N(C^*\cap D^*)| \leq b+m +c$.
But $b + m +d = |NC| =\kappa = |ND| = a +m +c$, and so $2\kappa = a + b + c +d +2m.$  Hence either  $|N(C\cap D)| < \kappa$ or
$|N(C^*\cap D^*)| < \kappa $ or $|N(C\cap D)| = |N(C^*\cap D^*)| = \kappa$. Whenever one of these boundaries has less or equal $\kappa$ elements then it separates $Y_1$ from $Y_2$. By the minimality of $\kappa $ we get $|N(C\cap D)| = |N(C^*\cap D^*)| = \kappa$. 
It follows that the opposite corners $C\cap D$ and $C^*\cap D^*$ have components that are cuts and so (A2) is satisfied.

Let $A$ and $B$ be components of opposite corners of $C$ and $D$ which contain elements of $\cc$. Then
\[|NA|+|NB|\le a + b + c +d +2m.\]
From before we know that $a + b + c +d +2m=2\kappa$. Since by definition $|NA|\ge\kappa$ and  $|NB|\ge\kappa$, we get $|NA|=|NB|=\kappa$. Thus $A$ and $B$ are in $\cc$ and Axiom $(A1)$ is proved.
\end{proof}

\begin{exam}\label{exam:3fan}
We define $X_n$ for $n\ge 3$ by $VX_n=\{a,b,c,d,1,2,\ldots,n\}$. There is a path $1,2,\ldots,n$, a circle $a,b,c,d,a$ and each of the vertices $1,2,\ldots,n$ is adjacent to each of the vertices $a,b$, see Figure~\ref{fig:finiteblock}. The graph is 2-connected, but there are no 2-inseparable sets which are separated by 2-element sets of vertices. But there are 3-inseparable sets (the sets $\{k,k+1,a,b\}$) which are pairwise separated from each other by 3-element sets. Hence $\kappa=3$. Set $C_k=\{1,2,\ldots,k-1\}$, $D_k=\{k+1,k+2,\ldots,n\}$.
The cut-system of Example~\ref{exam:finitecuts} is
\[\cc_n=\{C_k,D_k\mid 2\le k\le n-1\}.\]
A further discussion of this cut-system can be found in Example~\ref{exam:3fan_detail}.
\end{exam}
\vskip-4mm
\begin{figure}[htbp]
\centering
\begin{tikzpicture}[scale=0.35]    
    \path (3,10.3) coordinate (p0);
    \path (1.1,8.2) coordinate (p1);
    \path (0.1,5.7) coordinate (p2);
    \path (0,3.1) coordinate (p3);
    \path (0.9,1) coordinate (p4);
    \path (2.3,-0.5) coordinate (p5);
    \path (5,2) coordinate (p6);
    \path (5,7) coordinate (p7);
    \path (6.8,6.6) coordinate (p8);
    \path (6.8,2.3) coordinate (p9);
    \draw (p1)--(p7)-- (p0)-- (p1) -- (p2) -- (p3)-- (p4)-- (p5)-- (p6)-- (p9) -- (p8) --  (p7) -- (p6);
    \draw (p2) -- (p7) -- (p3)-- (p7)-- (p4)-- (p7) -- (p5);
    \path (intersection of p0--p6 and p1--p7) coordinate (0a);
 \filldraw[white] (0a) circle (3pt);
     \path (intersection of p0--p6 and p2--p7) coordinate (0b);
 \filldraw[white] (0b) circle (3pt);
     \path (intersection of p0--p6 and p3--p7) coordinate (0c);
 \filldraw[white] (0c) circle (3pt);
     \path (intersection of p0--p6 and p4--p7) coordinate (0d);
 \filldraw[white] (0d) circle (3pt);
     \path (intersection of p0--p6 and p5--p7) coordinate (0e);
 \filldraw[white] (0e) circle (3pt);
     \path (intersection of p1--p6 and p2--p7) coordinate (1a);
 \filldraw[white] (1a) circle (3pt);
     \path (intersection of p1--p6 and p3--p7) coordinate (1b);
 \filldraw[white] (1b) circle (3pt);
     \path (intersection of p1--p6 and p4--p7) coordinate (1c);
 \filldraw[white] (1c) circle (3pt);
     \path (intersection of p1--p6 and p5--p7) coordinate (1d);
 \filldraw[white] (1d) circle (3pt);
     \path (intersection of p2--p6 and p3--p7) coordinate (2a);
 \filldraw[white] (2a) circle (3pt);
     \path (intersection of p2--p6 and p4--p7) coordinate (2b);
 \filldraw[white] (2b) circle (3pt);
     \path (intersection of p2--p6 and p5--p7) coordinate (2b);
 \filldraw[white] (2b) circle (3pt);
     \path (intersection of p3--p6 and p4--p7) coordinate (3a);
 \filldraw[white] (3a) circle (3pt);
     \path (intersection of p3--p6 and p5--p7) coordinate (3a);
 \filldraw[white] (3a) circle (3pt);  
     \path (intersection of p4--p6 and p5--p7) coordinate (3a);
 \filldraw[white] (3a) circle (3pt);
 
   \draw (p0) -- (p6) -- (p1);
   \draw (p2) -- (p6) -- (p3);
   \draw (p4) -- (p6);
       
    \filldraw (p0) circle (3pt) ;
    \filldraw (p1) circle (3pt) ;
    \filldraw (p2) circle (3pt) ; 
    \filldraw (p3) circle (3pt) ;
    \filldraw (p4) circle (3pt) ;
    \filldraw (p5) circle (3pt) ;
    \filldraw (p6) circle (3pt) ; 
    \filldraw (p7) circle (3pt) ;
    \filldraw (p8) circle (3pt) ;
    \filldraw (p9) circle (3pt) ;      
       
\draw (p0) node [above] {1};
\draw (p1) node [above] {2\,\,};
\draw (p2) node [left] {3};
\draw (p5) node [below] {$n$};
\draw (p6) node [below] {$a$};
\draw (p7) node [above] {\ $b$};
\draw (p8) node [right] {$c$};
\draw (p9) node [right] {$d$};

\end{tikzpicture}
\vskip-1mm\caption{Finite 2-connected graph with 3-inseparable blocks}\label{fig:finiteblock}\vskip-5mm
\end{figure}
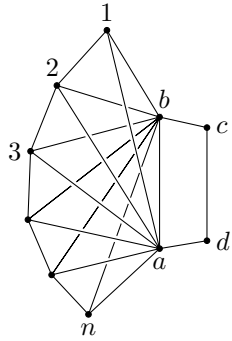

\section{Thin subsystems}

\begin {defi} \label {thin} Let $\cc$ be a cut system in a connected graph $X$. Let $\kappa$ be the smallest cardinality of a separator. A separator with $\kappa$ elements is a \emph{thin separator}. A cut $C$ in $\cc$ is called \emph{thin} if $|NC|=\kappa$.  A pre-cut $E$ is \emph{thin} if $E$ or $E^*$ is a thin cut.
\end {defi}

.

\begin{lemm}\label{lemm:intersection}
Let $C$ and $D$ be thin pre-cuts and assume that $C\cap D$ and $C^*\cap D^*$ contain a cut. Then $N(C\cap D)$ and $N(C^*\cap D^*)$ are thin separators. If $E$ is a cut component of $C\cap D$ then $NE=(C\cap ND)\cup (NC\cap ND)\cup (D\cap NC)$. Similarly, if $E$ is a cut component of $C^*\cap D^*$ then $NE=(C^*\cap ND)\cup (NC\cap ND)\cup (D^*\cap NC)$.

Moreover, $|C\cap ND|=|D\cap NC|$,  $|C^*\cap ND|=|D^*\cap NC|$ and $NC\cap ND=N(C\cap D)\cap N(C^*\cap D^*)$.  
\end{lemm}

%In \cite[Theorem 2]{Jung1977} and \cite[Proposition 2.1]{Jung1993}, Jung and Watkins prove a similar result. 

\begin{proof}
This is similar to the proof for Example~\ref{exam:finitecuts} and we again consider Figure~\ref{fig:cross}.
From the diagram, $\kappa = a + m + c = b + m + d$ and hence
\begin{equation}\label{equa:4n}
2\kappa = a +b +c +d +2m.
\end{equation}
Also
$$\openup2pt \displaylines{
|N(C\cap D)| \le a + d + m,\cr
|N(C\cap D^*)| \le a + b + m,\cr
|N(C^*\cap D^*)| \le b+c + m,\cr
|N(C^*\cap D)| \le c+ d + m. \cr}
$$
By Axiom (A1), the corners $C\cap D$ and $C^*\cap D^*$ both have at least one cut component. Hence $a + d + m\ge \kappa$ and $b+c + m\ge \kappa$. Now (\ref{equa:4n}) implies that these inequalities are equalities and $|N(C\cap D)|=|N(C^*\cap D^*)| =\kappa$. If $a<b$ then $d<c$, otherwise $a+m+c<\kappa$, and hence $a+m+d<\kappa$ which is impossible. If $b<a$ we get a contradiction in the same way and hence $a = b$, $c = d$. We saw that
\[N(C\cap D)= (C\cap ND)\cup (NC\cap ND)\cup (D\cap NC)\quad\mbox{and}\]
\[N(C^*\cap D^*)= (C^*\cap ND)\cup (NC\cap ND)\cup (D^*\cap NC).\]
From this we get $N(C\cap D)\cap N(C^*\cap D^*)=NC\cap ND$. If $E$ is a cut component of $C\cap D$ then $|NE|\ge \kappa$ and since $NE\subset N(C\cap D)$ we get $NE= N(C\cap D)$. Similarly, if $E$ is  cut component of $C^*\cap D^*$ then $NE=N(C^*\cap D^*)$.
\end{proof}

\begin{coro}
The thin cuts in a cut system form a cut system.
\end{coro}

\begin{proof}
Suppose that two opposite corners of two thin cuts contain thin cuts. By Axiom (A1) for the ambient cut-system, the components of the corners which contain these cuts, are cuts in the ambient system. Lemma~\ref{lemm:intersection} implies that these cuts are thin which implies Axiom (A1) for the set of thin cuts. Axiom (A2) for the set of thin cuts follows analogously.
\end{proof}

\begin {defi}\label {isolated} A corner of two sets of vertices is called \emph{isolated} if it does not contain a cut and if its adjacent links are both
 empty.\end{defi}
 \begin {defi} \label {nested}  We call two sets of vertices \emph{nested} if they have an isolated corner.
 \end {defi}
\begin{exam}\label{exam:3fanB}
Consider Example~\ref{exam:3fan}. Then the cuts $C=\{1,2,\ldots,k\}$ and $D=\{l,l+1,\ldots,n\}$ have $\{c,d\}$ as non-empty isolated corner for $l-k\le 2$.
\end{exam}

\begin{lemm}\label{lemm:all_links}
Let $E$ and $F$ be pre-cuts. If $E\subset F$ then $E\cap F^*$ is an empty isolated corner.

If a corner of two thin cuts does not contain any cut and one of its links is empty then both links are empty (and the corner is isolated).

Pairs of thin cuts have either (i) no empty link or (ii) two non-empty links and two empty links which are adjacent to an isolated corner or (iii) all four links are empty, two opposite corners contain a cut and at least one corner is empty. In case (i) the cuts are not nested, in the cases (ii) and (iii) they are.
\end{lemm}

\begin{proof}
If $E\subset F$ then $E\cap F^*$ and $E\cap NF$ are empty. A vertex $x$ in  $F^*\cap NE$ would have to be adjacent to a vertex in $(E\cap F^*)\cup (E\cap NF)$, because it cannot be adjacent to a vertex in $E\cap F$. Hence $F^*\cap NE$ is also empty and $E\cap F^*$ is an empty isolated corner.

Let $C$ and $D$ be thin cuts of a cut system. By (A2) there are two opposite corners $K_1,K_2$ which contain a cut.
We saw in the proof of Lemma \ref{lemm:intersection} that the two links of any corner other than $K_1,K_2$ must have the same number of elements. Thus if one of these two links is empty then so is the other.
In particular, it is not possible for  thin cuts to have exactly one or  three empty links.

If (iii) all links are empty then one corner has to be empty, otherwise one of the cuts would not be connected.
Suppose (ii) there are precisely two empty links. If these links are opposite or adjacent with an adjacent corner that contains a cut, then $m=|NE\cap NF|=\kappa$ which would imply that all four links are empty. Thus if there are precisely two empty links then they are adjacent and their adjacent corner is isolated.
\end{proof}

\begin {defi} \label {slice} We define a \emph{slice} (or \emph{$\cc$-slice}) to be a component of $VX \setminus S$ that is not a cut, where $S$ is a separator.
\end {defi}
\begin{coro}\label{coro:not_nested_corner}

If a thin cut-system has no slices then pre-cuts $E,F$ are nested if and only if
\[E\subset F,\quad E\subset F^*,\quad E^*\subset F\quad \emph{or}\quad E^*\subset F^*.\]
\end{coro}

\begin{proof}
That two pre-cuts are nested means that they have an isolated corner. If there are no slices then this corner is empty. Since its adjacent links are also empty, one of the inclusions hold.

If one of the inclusions holds then there is an empty corner with an empty adjacent link. Lemma~\ref{lemm:all_links} implies that the second adjacent link is also empty and hence this corner is isolated.
\end{proof}

In other papers, sets or vertices (or other sets) $E,F$ are usually called \emph{nested} if one of the following inclusions holds: $E\subset F$, $E\subset VX\setminus F$, $VX\setminus E\subset F$, $VX\setminus E\subset VX\setminus F$. It may happen for instance that $E\subset VX\setminus F$, but none of the inclusions of Corollary~\ref{coro:not_nested_corner} holds. For example, let $X$ be a cycle of length four, let $x,y$ be adjacent vertices and $E=\{x\}$, $F=\{y\}$. Thus being ``nested'' in the present paper is similar to the usual notation, but it is not a generalization.

\begin{lemm}\label{lemm:lemmathin}
If $E$ and $F$ are thin pre-cuts then there is a cut component of $E$ which contains $E\cap NF$.
\end{lemm}

\begin{proof}
Axiom (A2) implies that there is an $A$ in $\{E\cap F,E\cap F^*\}$ which contains a cut and whose opposite corner also contains a cut. By Axiom (A1) this corner $A$ has a cut component $C$ and Lemma~\ref{lemm:intersection} implies $|NA|=\kappa$. Now $NC\subset NA$ and $|NC|\ge\kappa$, because $C$ is a cut. This implies $NC=NA$. Let $E_0$ be the component of $E$ which contains $C$. Axiom (A1) applied to $E_0$ with itself implies that $E_0$ is a cut, because $E_0^*$ contains a cut. By Lemma~\ref{lemm:intersection}, every vertex $x$ in $E\cap NF$ is adjacent to $C$. Hence $x$ is in $E_0$ and $E\cap NF \subset E_0$. Note that we have not excluded the case $E\cap NF=\emptyset$.
\end{proof}

A cut is called an \emph{$A$-cut} if it is nested with all other thin cuts.
It is called a \emph{$B$-cut} if its $*$-complement has only one cut component.

\begin{theo}\label{theo:connectedness}
A thin cut is either an $A$-cut or a $B$-cut.
\end{theo}

\begin{proof}
Let $C$ be a thin cut which is not an $A$-cut. Then there is a thin cut $D$ which is not nested with $C$. By Lemma~\ref{lemm:lemmathin}, there is a cut component $C_0^*$ of $C^*$ which contains $C^*\cap ND$. In order to prove that $C$ is a $B$-cut we have to show that $C^*$ has no other cut component. 

Suppose there is another cut component $C_1^*$ of $C^*$. Then $NC_1^*$ is a separator. Hence $NC_1^*=NC^*$, otherwise $NC_1^*$ would have less than $\kappa$ elements. Now $C_1^*\cap C^*\cap ND=C_1^*\cap ND=\emptyset$, because, $C^*\cap ND\subset C_0^*$. There is an element $y$ in $D\cap NC^*$ and an element $z$ in $D^*\cap NC^*$, because, by case (i) of Lemma~\ref{lemm:all_links}, all links are not empty. Since $NC_1^*=NC$ and $C_1^*$ is connected, there is a path from $y$ to $z$ which is completely contained in $C_1^*$, except for its end-vertices $y$ and $z$. This path has to intersect $C^*\cap ND$ which contradicts $C_1^*\cap ND=\emptyset$.
\end{proof}

\begin{lemm}\label{lemm:slice_trivial}
For any cut system, components of isolated corners are slices.
\end{lemm}

\begin{proof}
Let $Q$ be a component of an isolated corner of $C$ and $D$. Then $NQ\subset NC\cap ND$ because the links of this corner are empty.
Since $C$ and $D$ are connected, $Q$ has to be a component of $C^*\cap D^*$. Hence $Q$ is a component of $C^*$ as well as a component of $D^*$, and $Q$ does not contain a cut, so it is not cut, by Axiom (A1).
\end{proof}

\begin{lemm}\label{lemm:slice}
Let $\cc$ be a thin cut system. A slice has empty intersection with each separator. Distinct slices are disjoint. If $Q$ is a slice, then no pair of elements of $NQ$ are separated by any separator.
\end{lemm}

\begin{proof}
Let $Q_1$, $Q_2$ be distinct slice components of $VX\setminus NC$ and $VX \setminus ND$ respectively where $C$, $D$ are thin cuts. By Lemma~\ref{lemm:lemmathin}, the links $ND\cap C$ and $ND\cap C^*$ are contained in cut components of $VX\setminus NC$. Hence they are disjoint from $Q_1$. It follows that $ND\cap Q_1=\emptyset$.

Consider corners and links of the pair  $Q_1$, $Q_2$.  We have shown that the links $Q_1\cap NQ_2$, $Q_2\cap NQ_1$ are empty. 
By the connectedness of $Q_1$, $Q_2$ this means that either $Q_1 = Q_2$ or $Q_1\cap Q_2 = \emptyset $.

Finally suppose $x, y \in NQ$ for some slice $Q$ and $x, y$ are separated by $NC$ for some cut $C$.  Now $Q$ is connected and
the path in $Q$ joining $x, y$ must intersect the separator $NC$, which we have already shown cannot occur.
\end{proof}

\begin{coro}
If a group $G$ acts transitively on a connected graph with a $G$-invariant thin cut system, then there are no slices.
\end{coro}

Let $X$ be a graph with a cut system of thin cuts. Lemma~\ref{lemm:slice} enables us to replace $X$ with another graph $\hat X$ with essentially the same cut
system, but in which there are no slices.   In this new cut system two cuts are nested if and only if there is an empty corner with empty adjacent links.

Let, then, $X$ be a graph with a cut system $\cc$ of thin cuts.
We define $\hat X$ as follows: $V\hat  X \subset VX$ and $v \in V\hat X$ if and only if $v\notin Q$ for every slice $Q$.
Note that if $v \in NC$ for some $C\in \cc$, then, by Lemma \ref{lemm:slice}, $v \in V\hat X$.
Two vertices $u,v \in V\hat X$ are joined by an edge in $\hat X$ if and only if
$u, v$ are joined by an edge in $X$ or if $u,v \in NQ$ for some slice $Q$.
We define $\hat \cc$ to be the set of sets $\hat C$, where $\hat C = C\cap V\hat X$ for some $C \in \cc$.

\begin{theo}\label{theo:hat}
Let $\cc$ be a thin cut system in $X$. Then $\hat X$ is connected and $\hat \cc$ is a cut system of thin cuts in $\hat X$. The separators of $\cc$ are the same as the separators of $\hat \cc$.
There are no slices in $\hat X$ with respect to $\hat \cc$.
\end{theo}

\begin{proof}
First we show that $N\hat C$ in $\hat X$ is the same as $NC$ in $X$, for all $C\in\cc$.
Suppose there is a vertex $x$ in $N\hat C\setminus NC$. This vertex is not contained in any $\cc$-slice in $X$. If $x$ were in $C$ then it would be in $\hat C$. So $x$ is in $C^*$. Since $x\in N\hat C$, there is a $y$ in $\hat C$ which is a  neighbour of $x$ in $\hat X$. Then $\{x,y\}$ is an edge in $\hat X$ but not an edge in $X$, because $y\in C$. This can only happen if $x$ and $y$ are in $NQ$ for some $\cc$-slice $Q$ in $X$. By Lemma~\ref{lemm:slice}, $Q\cap NC=\emptyset$. Since $y\in\hat C$ and $\hat C\subset C$ we have $Q\subset C$. Slices are connected, thus there is a path from $y$ to $x$ in $X$ which is contained in $Q$ except for its last vertex $x$. This is a path from a vertex in $C$ to a vertex in $C^*$ which is disjoint from $NC$, a contradiction.

Suppose there is a vertex $x$ in $NC\setminus N\hat C$. This vertex $x$ has a  neighbour $y$ in $C\setminus \hat C$. By the definition of $\hat C$, $y$ is contained in a slice component $Q$ of $VX\setminus ND$ for some $D\in\cc$.
If $NC=ND$ then $C$ and $Q$ would be distinct components in the complement of $NC=ND$, because $Q$ is a slice and $C$ is a cut. This is impossible since $y\in C\cap Q$. Hence $NC\ne ND$ and since $Q\subset C$ there must be a vertex $z$ in $NQ\cap C$. By Lemma~\ref{lemm:slice}, separators and slices are disjoint, hence $z$ is not contained in any slice and $z\in\hat C$. The vertices $x,z$ are both in $NQ$ and therefore adjacent in $\hat X$.
This is a contradiction, because $x$ is neither in $\hat C$ nor in $N\hat C$ but $z$ is in $\hat C$. Thus $NC=N\hat C$.

Let $x, y \in V\hat  X$.  Since $X$ is connected, there is a path 
$x = x_1, x_2, \dots , x_n = y$ from $x$ to $y$ in $X$.  Let
$x = x_1 = y_1, y_2, \dots , y_m =x_n =y$ be the subsequence of $x_1,x_2, \dots , x_n$ obtained by deleting the vertices that lie in any slice.
We will show that this subsequence 
is a path in $\hat X$ from $x$ to $y$.   If $y_i = x_j$  and $y_{i+1} = x_{j+1}$,   then $y_i$ is adjacent to $y_{i+1}$ in both $X$ and $\hat X$. 
If $y_i = x_j$ and $y_{i+1} = x_{j+r}$ where $r >1$, then $x_{j+1}, \dots , x_{j+r-1}$ all lie in the same slice $Q$. This is because adjacent vertices
cannot lie is distinct slices,  since distinct slices are disjoint and do not intersect any separator by Lemma~\ref{lemm:slice}. This means that $y_i$ and $y_{i+1}$
are in $NQ$ and so they are joined by an edge in $\hat X$.
Thus $\hat X$ is connected.

Next we show that $\hat\cc$ is a cut system. 
For $C\in\cc$, we need to show that $\hat C = C\cap V\hat X$ is connected in $\hat X$.  Any two points of $\hat C$ are joined by a path with vertices
in $C$. If we delete those vertices of the path which are contained in some slice then, as above, the resulting sequence will be a path in $\hat C$, because in $\hat X$ any pair of vertices in the boundary of some $\cc$-slice is joined by an edge.

Suppose there are opposite corners of $\hat C,\hat D\in \hat \cc$, for $C,D\in \cc$, which contain elements $C_0,D_0$ of $\hat C$, respectively. Axiom (A1) for $\cc$ implies that the components $A,B$ of the corresponding corners in $X$ which contain $C_0,D_0$, are in $\cc$.  Then $\hat A$ and $\hat B$ are in $\hat \cc$ which implies Axiom (A1) for $\hat\cc$. Axiom (A2) for $\hat\cc$ is proved analogously. The cut system $\hat\cc$ is thin because $\cc$ is thin and $N\hat C=NC$ for all $C\in\cc$.

Finally, we show that there are no slices in $\hat X$ with respect to $\hat \cc$. Let $K$ be a component of $V\hat X\setminus N\hat C$, for some $\hat C\in\hat\cc$. Let $K'$ be the component of $VX\setminus NC$ in $X$ which contains $K$. This component $K'$ cannot be a slice with respect to $\cc$ in $X$, because no vertex of a slice in $X$ is a vertex of $\hat X$. Hence $K'$ is a cut in $\cc$ and thus $K=K'\cap V\hat X$ is a cut in $\hat\cc$.
\end{proof}

\begin{coro}\label{coro:nested_hat}
Let $C$ and $D$ be cuts in a thin cut system $\cc$. Then the following are equivalent:

\begin{enumerate}
\item[(i)] $C$ and $D$ are nested in $\cc$,
\item[(ii)] $\hat C$ and $\hat D$ are nested in $\hat\cc$,
\item[(iii)] one of the inclusions $\hat C\subset \hat D$, $(\hat C)^*\subset \hat D$, $\hat C\subset (\hat D)^*$, $(\hat C)^*\subset (\hat D)^*$ holds.
\end{enumerate}
\end{coro}

\begin{proof}
This follows from Corollary~\ref{coro:not_nested_corner} and Theorem~\ref{theo:hat}.
\end{proof}

\begin{lemm}\label{lemm:not_nested_corner}
Let $C,D,E$ be thin cuts.
\begin{enumerate}
\item[(i)] If $E$ is nested with $D$ then either $E$ is nested with each cut component  of $C\cap D$ and of $C^*\cap D$  or with each cut component  of 
$C\cap D^*$ and of $C^*\cap D^*$.
\item[(ii)] If $E$ is  nested with  $C$ and $D$, then $E$ is nested with each cut component of any corner of $C$ and $D$.
\end{enumerate}
\end{lemm}

\begin{proof}
We can assume that there are no slices by replacing $X$ with $\hat X$ if necessary, see Corollary~\ref{coro:nested_hat}. Suppose $E$ is nested with $D$.

Case 1. Suppose $C,D$ are $A$-cuts. Then they have an empty isolated corner. Its adjacent corners are either $C$ or $D$ or the cut components of  these adjacent corners are cut components of $C^*$ or $D^*$. Hence the cut components of the corners which are adjacent to the isolated corner are $A$-cuts. Thus $E$ is nested with the cut components of the corners which are adjacent to the isolated corner.

Case 2. If $C,D$ are $B$-cuts then $C^*, D^*$ are cuts, because we are considering $\hat X$. By changing $D$ to $D^*$  if necessary
 we  may assume that $E\subset  D^*$ or $E^*\subset D^*$.    If $E\subset D^*$ then $E\subset C\cup D^*$ and $E\subset C^* \cup D^*$ and so $E$ is nested with any cut component of 
 the corners $C^*\cap D= (C\cup D^*)^* $  or $C\cap D$.  Similarly if $E^* \subset D^*$, then $E^*$ is nested with any cut component of the same two corners.  But $E$ is nested with a cut if and only if $E^*$ is nested with the same cut.
 
If $E$ is nested with both $C$ and $D$, then by (i) it will be nested with the cut components of three of the four corners.
 In fact it will be nested with the cut components of all four corners, because if say $E \subset D^*$ and $E\subset C^*$, then the above shows that $E$ is nested with
the cut components of $C\cap D$, $C\cap D^*$, $C^*\cap D$.  But $E \subset C^*\cap D^*$ and so it is nested with the cut components of the fourth corner, too.
\end{proof}

\section{Separation by finitely many cuts}\label{sect:finitely_many}

We call a set $S$ of vertices a \emph{tight $x$-$y$-separator} if $VX\setminus S$ has two distinct components $A,B$ which are adjacent to all elements of $S$ and for which $x\in A$ and $y\in B$.
Note that this is a general graph-theoretic definition which does not refer to our axiomatic cut systems and their separators.

The following can be found similarly in papers by Halin \cite[Statement 2.4]{Halin1991}, \cite[Corollary 1]{Halin1992} or Thomassen and Woess \cite[Proposition~4.1]{thomassen1993}.

\begin{lemm}\label{lemm:finitely_many}
For every integer $k$ and every pair $x,y$ of vertices in a connected graph there are only finitely many tight $x$-$y$-separators of order $k$.
\end{lemm}

\begin{proof}
We use induction on $k$.
Any tight $x$-$y$-separator of order $1$ is a vertex in any given path joining $x$ and $y$ and so the lemma is true for $k=1$.

Suppose the lemma holds for all tight $x$-$y$-separators of order $k$ in any connected graph $X$. We choose a path $\pi$ from $x$ to $y$ and assume that there are infinitely many tight $x$-$y$-separators of order $k+1$, $k+1\ge 2$. Then there is a vertex $z\in\pi\setminus\{x,y\}$ which is contained in infinitely many of these separators. If $S_1$ and $S_2$ are such $x$-$y$-separators of order $k+1$ then $S_1\setminus\{z\}$ and $S_2\setminus\{z\}$ are distinct tight $x$-$y$-separators of order $k$ in $X\setminus \{z\}$. Hence there are infinitely many distinct tight $x$-$y$-separators of order $k$ in $X\setminus \{z\}$ contradicting the induction hypothesis.
\end{proof}
\begin{lemm}\label{lemm:finite}
A thin pre-cut is nested with all but finitely many thin pre-cuts.
\end{lemm}

\begin{proof}
Let $E$ and $F$ be  thin pre-cuts which are not nested. By Lemma~\ref{lemm:all_links}
all links are not empty.  Hence $NF$ is a tight $x$-$y$-separator for two elements $x$ and $y$ of $NE$. Thin separators all have the same finite cardinality and Lemma~\ref{lemm:finitely_many} says that there are only finitely many such separators $NF$. Theorem~\ref{theo:connectedness} implies that there are just two cuts whose boundary is $NE$. The statement of the lemma follows.
\end{proof}

\section{Optimally nested cuts}\label{sect:opti}

Let $\cc$ be a cut system of thin cuts.
Let $C$ be a cut and let $M(C)$ be the set of thin cuts which are not  nested with $C$. Set $\mu  (C) = |M(C)|$. It follows from Lemma~\ref{lemm:finite}  that $\mu (C)$ is finite. 
If $D$ is a cut and $C_i$ is a cut component of $D^*$ then we put $\mu (D^*) = \mu (C_i)$.
There is no ambiguity in doing this as
if there are distinct cut components $C_1, C_2$ of $D^*$ then $N(D^*) = N(C_1) = N(C_2)$, and by Theorem~\ref{theo:connectedness},  $D, C_1, C_2$ are all $A$-cuts. Thus $\mu(D^*)=\mu(D)=\mu (C_1) = \mu (C_2)= 0$. We conclude, if $D$ is a thin cut and $C$ is a cut component of $D^*$ then $\mu(D)$, $\mu(D^*)$, $\mu(C)$ are all well defined and $\mu(D)=\mu(D^*)=\mu(C)$.

If $A,B$ are opposite corners which contain cuts then, by Lemma~\ref{lemm:intersection}, we can define $\mu(A)$ as $\mu(C)$ for any cut component  $C$ of $A$.

\begin{lemm}\label{lemm:corners_equality}
Let $C$ and $D$ be thin cuts and suppose that $C\cap D$ and $C^*\cap D^*$ contain cuts. Then
\begin{enumerate}
\item[(i)]
\[\mu(C\cap D) + \mu (C^*\cap D^*) \le \mu (C) + \mu (D).\]
\item[(ii)] If moreover, $C$ and $D$ are not nested then
\[\mu(C\cap D) + \mu (C^*\cap D^*) + 2 \le \mu (C) + \mu (D).\]
\end{enumerate}
\end{lemm}

\begin{proof}
If a thin cut $E$ is in $M(C^*\cap D^*)\cap M(C\cap D)$ then, by Lemma~\ref{lemm:not_nested_corner}~(i), $E$ is in $M(C)$ and in $M(D)$. Hence if $E$ is counted in twice on the left side of (i) then it is also counted twice on the right.

If $E$ is in $M(C\cap D)\setminus M(C^*\cap D^*)$ or in $M(C^*\cap D^*)\setminus M(C\cap D)$, that is $E$ is counted exactly once on the left side of (i), then, by Lemma~\ref{lemm:not_nested_corner}~(ii), $E$ is in $M(C)$ or in $M(D)$. Hence $E$ is counted at least once on the right side of (i) which establishes equation (i). If $C,D$ are not nested then $C$ and $D$ are counted on the right side, but not on the left side, by Lemma~\ref{lemm:all_links}. This implies (ii).
\end{proof}

Set $\mu_{\min} =\min\{\mu (C)\mid C \mbox{\ is a thin cut.}\}$. This minimum exists, because the values $\mu (C)$ are all finite.
\begin {defi} \label {optimal}  A thin cut $C$ with $\mu (C)=\mu_{\min}$ is called \emph{optimally nested}.
\end{defi}
 Every non-empty cut system contains an optimally nested thin cut.

\begin{theo}\label{theo:optimally}
Every optimally nested cut is nested with all other optimally nested cuts. Optimally nested cuts form a cut system.
\end{theo}

\begin{proof}
Optimally nested cuts are thin by definition. Theorem~\ref{theo:connectedness} says that for an optimally nested cut $C$ either $\mu_{\min}=0$ if $C$ is an A-cut or, if $C$ is a B-cut, then $C^*$ has precisely one cut component and $\mu(C^*)$ is well defined as $\mu(C)$.

Suppose there are optimally nested thin cuts $C$ and $D$ which are not  nested with each other.  By (A2) there will be opposite corners which contain cuts.   
By relabeling we can assume that these corners are $C\cap D$ and $C^*\cap D^*$ and by 
Lemma~\ref{lemm:intersection}, each of $C\cap D$ and $C^*\cap D^*$ has a component which is a  thin cut.
Now Lemma \ref{lemm:corners_equality} says that 
\[\mu (C\cap D) + \mu (C^*\cap D^*)< \mu (C) + \mu (D)  = 2\mu_{\min}. \]
Thus one of the summands on the left side of the inequality is less than $\mu_{\min}$, contradicting the minimality of $\mu_{\min}$.

If optimally nested cuts $C,D$ in a cut system $\cc$ have opposite corners which contain cuts then Axiom (A1) for $\cc$ says that any of their components which contains a cut is a cut. Lemma~\ref{lemm:corners_equality} (i) implies that these cuts are optimally nested. Hence Axiom (A1) is satisfied for the set of optimally nested cuts in $\cc$.
Similarly, Axiom (A2) for $\cc$ together with Lemma~\ref{lemm:corners_equality} (i) implies Axiom (A2) for the set of optimally nested cuts. Thus the optimally nested cuts in $\cc$ form a cut system.
\end{proof}

If $\mu_{\min} = 0$ then there are  thin 
cuts that are nested with every other thin cut.
It can happen though that $\mu_{\min} \not= 0$ as we show in an example.

 \begin{exam}\label{exam:2ended-graph}
 Let $X_n$ be the $2$-ended graph which is the induced subgraph of the integer lattice in the plane with $VX_n = \{ (i, j) \mid i \in \Z,\ j =0,1, 2, \dots ,n \}$.
Let $\ce$ be the cut system as in Example~\ref{exam:vertexends} restricted to thin cuts, that is cuts $C$ with $|NC|=n+1$. Then $C = \{ (i, j) \in VX_n \mid i> 0\} $ with $NC=\{(0,0),(0,1),\ldots (0,n)\}$ is an optimally nested cut.
If $n\geq 2$ then $D=\{(i,j)\mid j\le i\}$ with $ND = \{ (-1, 0), (0, 1),\ldots ,(n-1,n)\}$ is a thin cut which is not nested with $C$ and thus $\mu_{\min}>0$. That $C$ and $D$ are not nested can be verified by observing that they have no empty links because $(1,2)\in C\cap ND$, $(-1,0)\in C^*\cap ND$, $(0,0)\in D\cap NC$, $(0,2)\in D^*\cap NC$.
\end{exam}

Another graph for which $\mu_{\min} \not= 0$ is given in Figure~\ref{fig:4ends}.

\begin{prop}\label{prop:finitesequence}
Let $c$ be an integer. There is no strictly descending sequence of connected sets of vertices with non-empty intersection, whose boundaries have less than $c$ elements.
\end{prop}

\begin{proof}
Suppose there is such a sequence $(C_n)_{n\in\N}$. Let $I$ be the intersection $I = \bigcap \{ C_n | n =1,2,\dots \} $. If $u \in NI$, then
$u \in NC _n$ for all sufficiently large $n$.    Thus, since $|NC_n| < c$ we have $|NI| < c$ and $NI \subset NC_n$ for all sufficiently large $n$.
This must mean that $I$ is a oomponent of $C_n$ for these values of $n$.   But $C_n$ is connected and so the proposition follows.
\end{proof}

\begin{theo}\label{theo:mu=0}
Let $X$ be a graph with a thin cut system $\cc$ for which $\mu_{\min} \not= 0$.   Then $X$ is an infinite graph such that each cut contains an end.
\end{theo}

\begin{proof}
We know that for every $C\in \cc$ there is another cut $D$ with which it is not nested, and for every such pair $C, D$ there are
at least two corners which contain cuts.  Thus for any $C\in \cc$ we can form a sequence of distinct cuts $C = C_1, C_2, \dots  $, where $C_{i+1} \subset C_i$
and $C_{i+1}$ is a cut component of a corner of $C_i$ and another cut in $\cc$.  This sequence will determine an end of $X$ whose rays are all contained in $C$ from some index on, because the intersection of the elements of this sequence is empty by Proposition~\ref{prop:finitesequence}.
There will be another such sequence starting with $C^*$ and so $X$ has more than one end.
\end{proof}

We get a different proof of Theorem~\ref{theo:mu=0} in Section~\ref{sec:blocks}, see Corollary~\ref{coro:leaves}.

\section{Cuts and trees}\label{sec:cuts_trees}

A cut system is called \emph{nested} if its cuts are all pairwise nested. Theorem~\ref{theo:optimally} says that the optimally nested cuts in any given cut system form a nested cut system.
Our next goal is to construct a tree from a nested thin cut system $\cc$. For this we will need that for all $C_1,C_2$ in $\cc$ there are only finitely many $D$ in $\cc$ with
\begin{equation}\label{equa:finitely}
C_1\subset D\subset C_2.
\end{equation}
This is satisfied for systems where the cardinalities of the separators are bounded by some constant, see Proposition~\ref{prop:finitesequence}. In particular this holds for thin systems.   The construction below actually works more generally  for cut systems in which every cut $C$ has the same value for $|NC|$, which is
true for thin systems.    It may work for nested cut systems which do not satisfy this condition, but there are examples where it does not.

%If $\cc$ does not have property (\ref{equa:finitely}) then we will obtain a forest instead of a tree. 

We build a graph $T=(VT,ET)$ from a nested  thin cut system $\cc$ where $ET$ can be regarded as $\cc$ and $VT=\cb\cup\cs$, where $\cs$ is the set of separators.

We are going to define the elements of $\cb$  as equivalence classes of cuts. In the following Section~\ref{sec:blocks} we will see that we can equivalently define the elements of $\cb$ as blocks of vertices as discussed in the Examples~\ref{exam:1-connect} and \ref{exam:2-connect}. Before defining the equivalence relation on nested cut systems in general, let us consider an example.

For an edge $e$ in a digraph, let $o(e)$ denote the origin vertex and $t(e)$ the terminal vertex. Thus $e=(o(e),t(e))$.

\begin{exam}\label{exam:tree_blocks}
In Example~\ref{exam:1-connect} the cut system $\cc$ is the set of components which we obtain by removing any of the cut vertices. This is the special case of Example~\ref{exam:finitecuts} for $\kappa=1$. The edge set $ET$ of the tree $T$ can be regarded as the set $\cc$. If we consider $T$ as a directed tree then for $C \in \cc$ set $t( C) = NC=\{x\}$, where $x$ is a cut-point, and let $o(C)$ be the block
that is contained in $C\cup \{x\}$ so that $x\in o(C)$. Thus edges of $T$ point from vertices in $\cs$ towards vertices in $\cb$.

Two cuts are defined as equivalent if they are minimal with respect to containing a given block. Let $B$ be a block and  $\{x_i\mid i\in I\}$ be the set of cut point in $B$. Let $C_i$ be the component of $VX\setminus \{x_i\}$ which contains $B\setminus \{x_i\}$. Then $\{C_i\mid i\in I\}$ is an equivalence class which corresponds to the block $B$. This yields a one-to-one correspondence between the equivalence classes of cuts and the blocks.

We have seen further examples of such trees in Examples~\ref{exam:2-connect} and \ref{exam:farey}.
\end{exam}

Now in general, let  $\cc$ be a nested system of thin cuts.
By replacing $X$ by $\hat X$ if necessary, we may assume that there are no slices and hence nestedness is defined by inclusion as in Corollary~\ref{coro:nested_hat}.  %It is shown in \cite   or   \cite that a set of nested sets by inclusion forms the edge set of a directed tree $T = T(\cc)$.

We will show that the vertices can be partitioned into separators and {\it blocks}.
Let $\cs $ be the set of separators.      
We define $\cb $ as the set of equivalence classes for a particular equivalence relation $\sim $ on $\cc$.

\begin{lemm}\label{lemm:equivalence}
For $C,D$ in a nested  thin  cut system put $C\sim  D$ if  either
\begin{itemize}
\item[(i)]  $C=D$ or 
\item[(ii)] $C^* \varsubsetneq  D$ and if
$C^*\varsubsetneq E \subset  D$, for $E \in \cc$, then $E=D$.
\end{itemize}
Then $\sim $ is an equivalence relation on $\cc$.
\end{lemm}

\begin{proof}
Clearly $\sim $ is reflexive.

Suppose $C\sim D$ and $C\neq D$. If $D\not\sim C$ then there is a cut $E$ with $D^* \varsubsetneq  E\varsubsetneq C$, hence $C^*\varsubsetneq E^*\varsubsetneq D$ by Lemma~\ref{lemm:compl_incl}. There is an $x\in NC\cap E^*$. Let $E'$ be the cut component of $E^*$ which contains $x$. The cuts $E'$ and $C$ are nested and thus have an empty isolated corner. Since $NC\cap E'\ne\emptyset$, either $C\cap E'^*$ or $C^*\cap E'^*$ is this isolated corner.
If it is $C\cap E'^*$ then $C\subset E'$ and $E \varsubsetneq C$ implies $E \varsubsetneq E'$, which is impossible. Otherwise $C^*\varsubsetneq E'\varsubsetneq D$ in contradiction to $C\sim D$.

Next we show that $\sim$ is transitive. Suppose $E\sim  C$ and $C\sim  D$ and suppose that $C,D,E$ are distinct elements of $\cc$, then  $E^* \varsubsetneq  C$ and  $C^*
 \varsubsetneq D$. Since $D,E$ are nested, there are four possibilities. 
 \begin {itemize} 
 \item[{\it Case 1.}] If  $E \subset  D^*$ then $C^* \varsubsetneq  E$ implies $C^* \varsubsetneq D^*$, contradicting $D^* \varsubsetneq  C$.
 
 \item [{\it Case 2.}] If  $E\subset D$ then $C^*\subsetneq E\subset D$ and hence $D=E$ (by the second statement in~(ii)), which we have excluded. 
\item [{\it Case 3.}] If  $E^* \subset D^*$ then $D\subset E$ which is similar to {\it Case 2}.
\item [{\it Case 4.}]  $E^* \subset D$ then either $E\sim D$ or there is an $A\in\cc$ such that
\[E^* \varsubsetneq A\varsubsetneq  D.\]
We have again four cases, because  $A$ and $C$ are nested. 
\item [{\it Case 4.1.}] If $C \subset A$ then $C\varsubsetneq D$ contradicting $C^*\subset  D$. 
\item [{\it Case 4.2.}]
If $C^*\subset  A$ then $C^*\subset  A \varsubsetneq D$ and $C^*=A$,  because $C\sim D$. But then $E^*\varsubsetneq C^*$ contradicting $E^*\varsubsetneq C$.
\item [{\it Case 4.3.}] If  $C \subset  A^*$ then $A\subset C^*\subset E$ contradicting $E^*\subset A$.
\item [{\it Case 4.4.}]  If  $C^*\subset A^*$ then $E^*\varsubsetneq A\subset C$ and so $A=C$.  This  implies $C\varsubsetneq D$, contradicting $D^*\varsubsetneq C$.
\end {itemize}
\end{proof}

We obtain a directed graph $T=T(\cc)$
\[VT\ =\cs \cup\cb \quad\mbox{and}\quad ET=\cc\]
where  $\cs$ is the set of separators and  $\cb = \cc /\!\sim$.
Here  $o(C) =NC$ for  $C\in\cc$  and $t(C)=[C]$, where $[C]$ denotes the $\sim$-class which contains $C$. Clearly $T$ is a bipartite graph.    Each vertex in $\cb$ has every  adjacent edge pointing
towards it and each  vertex in $\cs$ has every adjacent directed edge  pointing away from it.  Any path in $T$ will have edges with alternating orientations as one proceeds along it. 
The elements of $\cb$ are sometimes called \emph{black vertices} and the ones in $\cs$ \emph{white vertices} and we draw them black and white, respectively, see Figures~\ref{fig:1block} to \ref{fig:Farey}, \ref{fig:4Exam}, \ref{fig:9Exam} to \ref{fig:dragon}, and \ref{fig:Tutte3}.

In fact $T$ is an oriented tree.    Indeed  if $E, F, G$ are distinct cuts such that $E\sim F$ and $NF =NG$,  so that $E, F$ are adjacent edges and 
$F, G$ are also adjacent edges, then $E^* \subsetneq G^*$, since we know that $E^* \subseteq F$ and $NF = NG$ means that $F$ is a component of $G^*$.
Since $E^* \subsetneq G^*,  G \subsetneq E$, it follows that in any path in $T$ the set of alternate edges correspond to cuts that are properly
ordered by inclusion.   Thus  the underlying graph of $T$ has no loops and so $T$ is an oriented tree.

We have proved the following.

\begin{theo}\label{theo:tree}

Let  $\cc$ be a nested   thin  cut  system.   Then $T=T(\cc)$ is an oriented  tree.
\end{theo}

In the following example, illustrated in Figure~\ref{fig:4Exam}, we see a cut system as in Example~\ref{exam:vertexends} with a non-empty slice similar to Examples~\ref{exam:3fan}, \ref{exam:3fanB} and \ref{exam:3fan_detail} and the corresponding structure tree.

\begin{exam}
Set $VX=\Z\cup \{o,i\}$ and
\[EX=\{\{x,x+1\}\mid x\in\Z\}\cup\{\{x,o\}\mid x\in\Z\}\cup\{\{o,i\}\}.\]
We consider the cut system of Example~\ref{exam:vertexends}.
The minimal number of vertices needed to separate the two ends of $X$ is $2$. The (connected) thin cuts are of the form $C^+_k=\{k+1,k+2,\ldots\}$ or $C^-_k=\{k-1,k-2,\ldots\}$, where $NC^+_k=NC^-_k=\{k,o\}$. The group of automorphisms acts transitively on the cut system $\cc=\{C^+_k,C^-_k\mid k\in\Z\}$.  If $l> r$ then $C^-_l$ and $C^+_r$ do not have an empty corner but they are nested because $\{i\}$ is their isolated corner.   The set $\{i\}$ is the only slice.   The graph $\hat X$ is obtained by deleting $i$.
The tree $T(\cc )$ is a double ray.

\end{exam}

%\vskip-5mm
\begin{figure}[htbp]
\centering
\begin{tikzpicture}[scale=0.5]
    \path (0,0) coordinate (p0);     
        \path (1,0) coordinate (p7);  
         \path (8,0) coordinate (p8);           
    \path (2,0) coordinate (p1);
    \path (3,0) coordinate (p2);
    \path (4,0) coordinate (p3);
    \path (5,0) coordinate (p4);
    \path (6,0) coordinate (p5);
  \path (7,0) coordinate (p6);
\path (4,2) coordinate (q);
    \path (4,4) coordinate (r);
    \draw (p1) -- (p5);
\draw [dashed] (p0) --(p1);
\draw [dashed] (p5) --(p8);
\draw [dashed] (p0)--(q)--(p8);
\draw [dashed] (q) -- (p7);
\draw [dashed] (q) -- (p6);
  \filldraw (p7) circle (3pt) ;
\filldraw (p0) circle (3pt) ;
 \filldraw (p8) circle (3pt) ;

  \filldraw (p1) circle (3pt) ;
  \filldraw (p2) circle (3pt);
  \filldraw (p3) circle (3pt);   
  \filldraw (p4) circle (3pt);
  \filldraw (p5) circle (3pt) ;
  \filldraw (q) circle (3pt);  
 \filldraw (r) circle (3pt);
\draw (p1) -- (q) -- (p2);
     \draw (p3) -- (q) -- (p4);
 \draw (p5) -- (q) ;
\draw (q) -- (r) ;
      \path (12,0) coordinate (p1);
    \path (13,0) coordinate (p2);
    \path (14,0) coordinate (p3);
    \path (15,0) coordinate (p4);
   \path (15.5,0) coordinate (p5);
   \path (9,0) coordinate (p6);
\path (11.5, 0) coordinate (r);
     \draw (p1) -- (p4);

     \path (12.5,0) coordinate (q1);
    \path (13.5,0) coordinate (q2);
    \path (14.5,0) coordinate (q3);
    \filldraw (p1) circle (3pt) ;
   
     \filldraw (p2) circle (3pt);
  \filldraw (p3) circle (3pt);  
   
 \filldraw (q1) [white]  circle (3pt) ;
   \draw (q1) circle (3pt) ;
   
     \filldraw (q2) [white] circle (3pt);
  \filldraw (q3) [white] circle (3pt);  
     \draw (q2) circle (3pt);
  \draw (q3) circle (3pt);  
\draw [dashed] (p1) -- (11, 0);
 \draw [dashed] (p4) -- (16, 0);
  \filldraw (p4) circle (3pt);
  \filldraw (q) circle (3pt);  
  \filldraw (p6) [white] circle (2pt);
   \filldraw (r) [white] circle (3pt);
  \draw (r) circle (3pt);  
  \filldraw (p5) [white] circle (3pt);
  \draw (p5) circle (3pt) ;
\path (5.8, 2) coordinate (q);
    \path (5.8, 4) coordinate (r);
\draw node at (3.5,2.4) {$o$};
\draw node at (3.5,4) {$i$};
   \draw (4.5,-1) coordinate (p1);
\draw (p1) node [left ] {$X$};

  \draw (14.5, -1) coordinate (p1);
\draw (p1) node [left] {$T(\cc )$ };

\end{tikzpicture}
\vskip-1mm\caption{Structure tree for a two-ended graph}\label{fig:4Exam}
\end{figure}
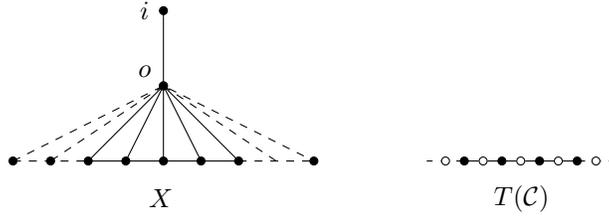

Let $\cn $ be the  nested system of the optimally nested thin cuts in a cut system as described in Section~\ref{sect:opti}.
If a group $G$ acts on $X$ and $\cc$ is invariant under $G$ then $G$ naturally acts on $T$.

\section{Blocks}\label{sec:blocks}

It is possible to define the set $\cb$ in a different - possibly better - way.
Let $\cc$ be a nested  thin cut system.  %For the moment we do not assume $\cc$ is nested.
A subset $Y$ of $VX$ is said to be \emph{$\cc$-inseparable} if for every $C\in \cc$  either  $Y \subset C\cup NC$ or $Y \subset C^*\cup NC$
but not both.
Every $\cc$-inseparable set is contained in a
maximal $\cc$-inseparable set.
Thus if one has an increasing sequence of $\cc$-inseparable sets $Y_1 \subset Y_2 \subset \dots $ and $Y = \bigcup  _n Y_n $,
then $Y$ is $\cc$-inseparable, since if $C \in \cc$ and $n$ is a positive integer such that $Y_n$ 
has more than $|NC|$ elements then for all $m\geq n$ either $Y_m\subset  C\cup NC$ or $Y_m\subset  C^*\cup NC$ and the same is true for $Y$.
\begin {defi} \label {block} A maximal $\cc$-inseparable set which is disjoint with all slices is called a \emph{$\cc$-block}. \end {defi}
Note that in this section we will use this definition of block or $\cc $-block, while a block as in the previous section is an element of $\cc /{\sim} $.
We can regard a block as a subgraph of $X$ if we say that an edge is in a block if both its vertices are.
We have seen that if $X$ is a connected graph with a cut system $\cc $, then we can replace $X$ by $\hat X$ and $\cc $ by $\hat \cc $ in which there are no slices.   The $\cc$-blocks become $\hat \cc$-blocks  in $\hat X$.    In the rest of this section (apart from Example~\ref{exam:3fan_detail}) we will assume that there are no slices in $X$, i.e.
if $A \in \cc $ then each component of $X\setminus NA$ is a cut.
% We will also assume that $\cc$ is a nested system.
 
For $b\in\cc/{\sim}$ we define
\[B(b)=\bigcap_{C\in b}(C\cup NC).\]

\begin{lemm}\label{lemm:blocks}
An edge of $X$ which is not contained in any separator is contained in exactly one block.
The set of all blocks is $\{B(b)\mid b\in\cc/{\sim}\}$.
If $b\in\cc/{\sim}$ then
\begin{equation}\label{equa:inclu}\bigcup_{C\in b}NC\subset B(b).
\end{equation}
If $C\in b$ then $B(b)$ is the only block $B$ such that $NC\subset B\subset C\cup NC$. Moreover, $B(b)\setminus NC\ne\emptyset$. 
\end{lemm}

\begin{proof}
The first statement follows from the fact that the vertices of an edge cannot be separated by a separator.

Let $B$ be a block. We show that all cuts $C$   which are minimal with respect to  $B\subset C\cup NC$ are equivalent.  There are minimal such cuts by Proposition
~\ref {prop:finitesequence}.  If $C\not= D$ are cuts and $B \subset C\cup NC, B\subset D\cup ND$,  then either $C^*  \subset D$ or $D^* \subset C$ since
$C, D$ are nested.   If now $E$ is a cut and $C^*\varsupsetneq  E \subset D$, then either $B \subset E\cup NE$ or $B \subset E^*\cup NE$.
If $B \subset E\cup NE$ then $D = E$ if $D$ is minimal with respect to $B \subset D\cup ND$, and if $B\subset E^*\cup NE$,  then $B \subset E^*\cup NE \subset  C\cup NC$.  This latter inclusion cannot happen if $C$ is minimal with respect to $B\subset C\cup NC$ and $E \not= C$.
 If $b$ is the corresponding equivalence class then $B=B(b)$. Thus every block $B$ occurs as block $B(b)$ of some equivalence class $b$.

Let $C\in b$ and suppose two vertices $x,y\in B(b)$ are separated by some separator $S$. Then $S\subset C\cup NC$, because $C$ is nested with the cuts $D$ for which $ND=S$. One of these cuts $D$ contains $C^*$. Let $D'$ be the cut such that $C^*\subset D'\subset D$ and $D'\sim C$. Either $x$ or $y$ is not in $D'$ and hence not in $B(b)$, a contradiction. Thus $B(b)$ is inseparable. On the other hand, any vertex which is not in $B(b)$ can be separated by some separator from $B(b)$. Hence $B(b)$ is a block.

If $D\sim C$ then $C^*\subset D$, $C^*\cup NC\subset D\cup ND$. This implies $NC\subset B$ and (\ref{equa:inclu}). The inclusion $B\subset C\cup NC$ follows from the definition of $B(b)$.

Suppose there is a block $B'$, $B'\ne B$, such that $NC\subset B'\subset C\cup NC$. There is a separator which separates $B$ from $B'$. Since both $B$ and $B'$ contain $NC$, this separator is $NC$. But then one of these two blocks has to be in $C\cup NC$ and the other in $C^*\cup NC$, a contradiction.

If $b$ contains two different cuts $C,D$, then $NC\ne ND$, by the definition of $\sim$, and (\ref{equa:inclu}) implies $B(b)\setminus NC\ne\emptyset$. If $b$ only contains one cut $C$ then $B(b)=C\cup NC$ and again $B(b)\setminus NC\ne\emptyset$.
\end{proof}

\begin{coro}
In a nested cut system, every block has at least $\kappa+1$ elements.
\end{coro}

We saw that $b\mapsto B(b)$ defines a bijection from $\cc/{\sim}$ to the set of all blocks. If there are no ambiguities from the context we may now consider the set $\cb$ as the set of $\cc$-blocks instead of $\cc/{\sim}$. For $C\in\cc$ let $b(C)$ denote the $\sim$-class which contains $C$. In the tree $T(\cc)$ we now have $t(C)=B(b(C))$. A separator $S$ and a block $B$ are adjacent as vertices in $T(\cc)$ if and only if $S\subset B$, see Lemma~\ref{lemm:blocks}. This construction of $T(\cc)$ with blocks implies the following.

\begin{coro}\label{coro:leaves}
Let $\cc$ be a thin nested cut system. Every leaf (i.e.\ vertex with degree one) of $T(\cc)$ is a block $B$ which is adjacent in $T(\cc)$ to a separator $S$ such that $B\setminus S$ is a cut.

%Let $\cc$ be an arbitrary thin system and $\cn$ its nested subsystem of optimally nested cuts. If $T(\cn)$ has a leaf then $\mu_{\min}=0$ in $\cc$.
\end{coro}

%\begin{proof}
%A cut in $\cc$ which is a block and a leaf in $T(\cn)$ is nested with all other cuts in $\cc$.
%\end{proof}

%Corollary~\ref{coro:leaves} together with Theorem~\ref{theo:optimally} yields another proof of Theorem~\ref{theo:mu=0}.
 
The induced  subgraph on a block $B$ is not usually connected, as can be seen from examples in Section 9.    However one obtains
a connected graph $X_B$ on $B$ by adding in {\it ideal} edges, which are edges joining each pair of vertices which lie in the same separator.

\begin {prop}
The graph $X_B$ is connected.
\end{prop}
\begin {proof}
The proof that $X_B$  is connected is similar to the proof that $\hat X$ is connected in  Theorem~\ref{theo:hat}.  

Let $x, y \in B $.  Since $X$ is connected, there is a path 
$x = x_1, x_2, \dots , x_n = y$ from $x$ to $y$ in $X$.  Let
$x = x_1 = y_1, y_2, \dots , y_m =x_n =y$ be the subsequence of $x_1,x_2, \dots , x_n$ obtained by deleting the vertices that do not lie in $B$.
We will show that this subsequence 
is a path in $X_B$ from $x$ to $y$.  If $C\sim D$ are in the equivalence class $[B]$, then $C^*\cap D^* = \emptyset$ since $C^*\subset D$.   Thus
the sets $C^*$ as $C$ ranges over $[B]$ partition the set  $VX \setminus B$.    
If $y_i = x_j$  and $y_{i+1} = x_{j+1}$,   then $y_i$ is adjacent to $y_{i+1}$ in both $X$ and $X_B$.

If $y_i = x_j$ and $y_{i+1} = x_{j+r}$ where $r >1$, then $x_{j+1}, \dots , x_{j+r-1}$ all lie in the same cut  $C^*$. This is because adjacent vertices
cannot lie is distinct $C^*$.  This means that $y_i$ and $y_{i+1}$
are in $NC^*$ and so they are joined by an edge in $X_B$.
Thus $X_B$ is connected.
\end {proof}

As mentioned in the introduction, we can add even more ideal edges and  join any two vertices by an (ideal) edge if they are not separated by any cut in our given nested cut system. The resulting cut system remains the same, but then every block spans a complete subgraph. This is illustrated in Figure~\ref {fig:2block}b.

\begin{exam}\label{exam:3fan_detail}
Consider the graph $X_n$ from Example~\ref{exam:3fan}. There are $n-2$ separators
$\{2,a,b\}, \{3,a,b\},\ldots ,\{n-1,a,b\}$. The set of $\cc_n$-blocks is
\[\cb_n=\{\{1,2,a,b\},\{2,3,a,b\},\ldots,\{n-1,n,a,b\}\}.\]
Instead of the blocks we could consider the $\sim$-classes
\[\cb_n\ \,=\ \cc_n/\sim\ =\
\{\{C_2\},\{D_2,C_3\},\{D_3,C_4\},\ldots,\{D_{n-2},C_{n-1}\},\{D_{n-1}\}\}.\]
The corresponding tree $T(\cc_n)$ is an alternating path of length $2n-2$ with $n-2$ white and $n-1$ black vertices. The system $\cc_n$ is thin.

For $l< k$ the intersection $C_k^*\cap D_l^*=\{c,d\}$ is a non-empty isolated corner, it is the only slice of $\cc_n$. For other pairs of cuts we have an empty isolated corner. Hence $\cc_n$ is nested. The graph $\hat{X}_n$ is obtained by deleting the vertices $c,d$.
\end{exam}

\section{Structure trees and nested cut-systems}

We have shown the following main theorem.

\begin{theo} \label{theo:main}
Let $X$ be a connected graph with a cut system $\cc$ invariant under a   group $G$  of automorphisms of $X$.   The set $\cn$ of optimally nested thin cuts in $\cc$ forms
the edge set of a $G$-tree  $T(\cc)$.
\end{theo} 
Let us now investigate further  properties of structure trees.
Let $G$ be a group acting on a connected graph $X$ and let $\cn$ be a nested system of thin cuts,
invariant under $G$. The action of $G$ on $\cn$ induces an action on $T=T(\cn)$ and hence $T$ is a $G$-tree.

We now show how  a ray $R$ in $X$ determines either a unique end or a unique vertex of $T$.   An end in $T$ will correspond to a unique ray
in $T$ starting at a fixed vertex.   We have seen that $T$ is oriented so that the orientations alternate on any path.
We can  choose either the even or odd edges of the ray so that they form a strictly decreasing sequence
of cuts $E_1 \supset E_2 \supset \dots $, which will have empty intersection in  $\hat X$ or intersection  which is a subset of a slice in $X$.

Let $R$ be a ray in $X$.    If there is a ray in $T$ as above for which for each $i$,  $R$ is eventually in $E_i$, then the end of $T$ with this property
is uniquely determined.      There may be no such  end of $T$. It will always be the case that for each cut $E$ in $\cn $ the ray is eventually 
in $E$ or in $E^*$.     The edges of $T$ which eventually contain $R$ will either `point at' an end of $T$ or to a  vertex,
and this end or block is uniquely determined.
This vertex is either a separator, in which case the ray lies in a slice from some index on, or the vertex is a block $B$. In the latter case
the ray contains infinitely many vertices of $B$.
We then say that the ray $R $ \emph{belongs to} the corresponding vertex or end of $T$.
It is clear  that the following holds.
 
 \begin{prop}
If  two rays belong to the same end of $T$ then  they belong to the same end of $X$.  If two rays belong to the same end of $X$, then they either
belong to the same end of $T$ or to the same vertex.
\end{prop}
 If $k\geq \kappa $, and if  $Y$ is a $k$-inseparable set, then there is a unique block $B$ that contains $Y$.  
 However it may be the case that a block contains no $k$-inseparable set.   Also a block may properly contain a maximal $k$-inseparable set. 
 See Figure \ref {fig:4blocks}.  The right hand block (with $6$ marked vertices)  is not $3$-inseparable, and the shaded block does not contain a $3$-inseparable set.

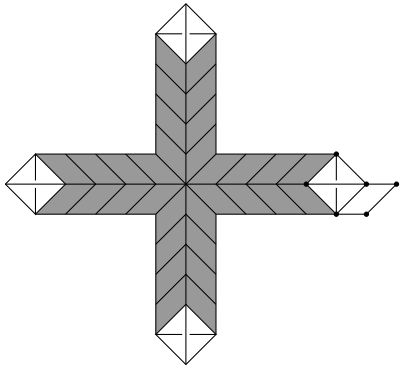
\begin{figure}[htbp]
\centering
\begin{tikzpicture}[scale=0.4]
 \fill[black!40] (0,6)--(1,5)--(0,4)--(4,4)--(4,0)--(5,1)--(6,0)--(6,4)--(10,4)--(9,5)--(10,6)--(6,6)--(6,10)--(5,9)--(4,10)--(4,6)--(0,6) ;
\draw (-1,5) --(0,4) --(0, 6)--(-1,5)  ;
\filldraw [white]  (0,5) circle (3pt);
\draw (11,5) --(10,4) --(10, 6)--(11,5)  ;

\filldraw [white]  (10,5) circle (3pt);
\draw  (11,5) --(12,5) -- (11,4) --(10,4) ;
\draw (5,11) --(4,10) --(6,10)--(5,11)  ;
\filldraw [white]  (5,10) circle (3pt);
\draw (5,-1) --(4,0) --(6,0)--(5,-1)  ;
\filldraw [white]  (5,0) circle (3pt);
 \filldraw (11,5) circle (2pt);
 \filldraw(12,5) circle (2pt);
 \filldraw(10,4) circle (2pt);
\filldraw(11,4) circle (2pt);
\filldraw(9,5) circle (2pt);
\filldraw(10,6) circle (2pt);

\draw (-1,5) -- (11,5) ;
\draw (5,11) -- (5,-1) ;
\draw (0,6) -- (4,6) ;
\draw (6,6) -- (10,6) ;
\draw (4,10) -- (4,6) ;
\draw (4,0) -- (4,4) ;
\draw (0,4) -- (4,4) ;
\draw (6,4) -- (10,4) ;
\draw (6,10) -- (6,6) ;
\draw (6,0) -- (6,4) ;
\draw (4,4) --(6,6) ;
\draw (6,4)--(4,6);
\draw (3,6)--(4,5) --(3,4) ;
\draw (2,6)--(3,5) --(2,4) ;
\draw (1,6)--(2,5) --(1,4) ;
\draw (0,6)--(1,5) --(0,4) ;
\draw (10,6)--(9,5) --(10,4) ;
\draw (9,6)--(8,5) --(9,4) ;
\draw (8,6)--(7,5) --(8,4) ;
\draw (7,6)--(6,5) --(7,4) ;
\draw (4,10)--(5,9) --(6,10) ;
\draw (4,9)--(5,8) --(6,9) ;
\draw (4,8)--(5,7) --(6,8) ;
\draw (4,7)--(5,6) --(6,7) ;
\draw (4,0)--(5,1) --(6,0) ;
\draw (4,1)--(5,2) --(6,1) ;
\draw (4,2)--(5,3) --(6,2) ;
\draw (4,3)--(5,4) --(6,3) ;
\end{tikzpicture}
\caption{Graph with a block containing no $3$-inseparable set}\label{fig:4blocks}\vskip-2mm
\end{figure}

In Example ~\ref{exam:finitecuts} for cut system $\ci $ take $\Omega $ to be the set of $k$-inseparable sets, and in   Example~\ref{exam:vertexends} for cut system 
$\ce $, take
$\Omega $ to be the set of vertex ends.   The thin subsystem of $\ci $ and the whole cut system $\ce $ are both examples of a thin cut system $\cc $
in which a cut $A$ is  in $\cc $ if it is thin (i.e $|NA| = \kappa $) and it separates a pair of elements of $\Omega $.     One 
could, in fact, get a cut system in which $\Omega $ is the union of  the set of vertex ends and the set of $k$-inseparable sets.

\begin {defi} \label {omega} An $\Omega $-cut system is a thin cut system $\cc$ in which a cut $A$ is in $\cc $ if it separates a pair of elements of $\Omega $.   Here $\Omega $ is either the set of vertex ends or the set of $\kappa $-inseparable sets or the union of these two sets.

\end {defi}
We will strengthen our main theorem for such a cut system to show that there is a uniquely determined nested subsystem (and therefore a structure tree) that separates any pair of elements of $\Omega $ if they are separated by a cut in $\cc$.

Note that this will be  a genuine strengthening as the following example shows.

\begin{exam} In this example there are distinct maximal  $\kappa$-inseparable sets are not separated by optimally nested cuts.   In Figure~\ref{fig:opt}  there are four $4$-inseparable sets (each of which is the vertex set of a complete
 subgraph on $5$ vertices).   There are two thin separators, shown in black, that correspond to optimally nested cuts $C$ with $\mu (C)=0$.
 The two central $4$-inseparable sets $Y_1, Y_2$ are not separated by any cut for which $\mu (C)=0$.   The  two separators, shown in grey,
 which separate $Y_1, Y_2$ correspond to cuts $C$ with $\mu (C) = 16$.   Any cut $D$ which separates $Y_1, Y_2$ has $\mu (D) \geq 16$.
 
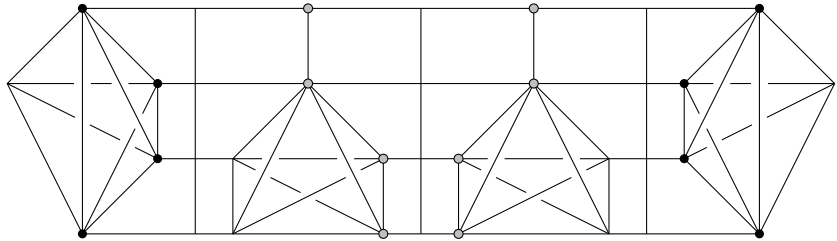
\begin{figure}[htbp]
\centering
\begin{tikzpicture}
\draw (5,6) -- (5,5) ;
\draw (2,6) --(2,5) ;
 \path (-1,6) coordinate (A);
    \path (-2,5) coordinate (B);
    \path (0,5) coordinate (C);
    \path (0,4) coordinate (D);    
    \path (-1,3) coordinate (E);    
    \draw (B)--(D) ;    
     \path (intersection of A--E and B--C) coordinate (Y);
      \path (intersection of A--E and B--D) coordinate (Z);
 \path (intersection of C--E and B--D) coordinate (X);
  \path (intersection of A--D and B--C) coordinate (W);
 \path (intersection of A--D and C--E) coordinate (U);

  \draw (B)--(C) ;
   
    \draw (A)--(B)--(E)--(D)--(C)--(A);
   \filldraw[white] (X) circle (3pt);
 \filldraw[white] (Y) circle (3pt);
 \filldraw[white] (Z) circle (3pt);
    \draw (E)--(C);

 \filldraw[white] (U) circle (3pt);
 \filldraw[white] (W) circle (3pt);
 
\draw (A)--(D) ;
\draw (A)--(E) ;

 \path (8,6) coordinate (A);
    \path (9,5) coordinate (B);
    \path (7,5) coordinate (C);
    \path (7,4) coordinate (D);    
    \path (8,3) coordinate (E);    
 
    \draw (B)--(D) ;
     \path (intersection of A--E and B--C) coordinate (Y);
      \path (intersection of A--E and B--D) coordinate (Z);
 \path (intersection of C--E and B--D) coordinate (X);
  \path (intersection of A--D and B--C) coordinate (W);
 \path (intersection of A--D and C--E) coordinate (U);

  \draw (B)--(C) ;
  
    \draw (A)--(B)--(E)--(D)--(C)--(A);
   \filldraw[white] (X) circle (3pt);
 \filldraw[white] (Y) circle (3pt);
 \filldraw[white] (Z) circle (3pt);
    \draw (E)--(C);

 \filldraw[white] (U) circle (3pt);
 \filldraw[white] (W) circle (3pt);
   \draw (A)--(D) ;

\draw (A)--(E) ;

 \path (2,5) coordinate (A);
    \path (1,4) coordinate (B);
    \path (3,4) coordinate (C);
    \path (3,3) coordinate (D);    
    \path (1,3) coordinate (E);

   \draw (B)--(D) ;
     \path (intersection of A--E and B--C) coordinate (Y);
      \path (intersection of A--E and B--D) coordinate (Z);
 \path (intersection of C--E and B--D) coordinate (X);
  \path (intersection of A--D and B--C) coordinate (W);
 \path (intersection of A--D and C--E) coordinate (U);

  \draw (B)--(C) ;
    \draw (A)--(B)--(E)--(D)--(C)--(A);
   \filldraw[white] (X) circle (3pt);
 \filldraw[white] (Y) circle (3pt);
 \filldraw[white] (Z) circle (3pt);
    \draw (E)--(C);

 \filldraw[white] (U) circle (3pt);
 \filldraw[white] (W) circle (3pt);
 \draw (A)--(D) ;
\draw (A)--(E) ;
\path (5,5) coordinate (A);
    \path (4,4) coordinate (B);
    \path (6,4) coordinate (C);
    \path (6,3) coordinate (D);    
    \path (4,3) coordinate (E);

   \draw (B)--(D) ;
     \path (intersection of A--E and B--C) coordinate (Y);
      \path (intersection of A--E and B--D) coordinate (Z);
 \path (intersection of C--E and B--D) coordinate (X);
  \path (intersection of A--D and B--C) coordinate (W);
 \path (intersection of A--D and C--E) coordinate (U);

  \draw (B)--(C) ;
    \draw (A)--(B)--(E)--(D)--(C)--(A);
   \filldraw[white] (X) circle (3pt);
 \filldraw[white] (Y) circle (3pt);
 \filldraw[white] (Z) circle (3pt);
    \draw (E)--(C);

 \filldraw[white] (U) circle (3pt);
 \filldraw[white] (W) circle (3pt);
 \draw (A)--(D) ;
\draw (A)--(E) ;

\draw (-1, 3) -- (1,3) ;
\draw (3, 3) -- (4,3) ;
\draw (6, 3) -- (8,3) ;

\draw (0,5) -- (7,5) ;
\draw (-1,6) -- (8,6) ;
\draw (6.5,6) -- (6.5,3) ;
\draw (0,4) -- (1,4) ;
\draw (3,4) -- (4,4) ;
\draw (7,4)--(6,4) ;

\draw (3.5,6)--(3.5,3) ;

\draw (.5,6)--(.5,3) ;

 \filldraw[black!100] (0,5) circle (1.5pt);
 \filldraw[black!100] (-1,6) circle (1.5pt);
 \filldraw[black!100] (0,4) circle (1.5pt);
 \filldraw[black!100] (-1,3) circle (1.5pt);
 \draw[black,fill=black!25] (5,5) circle (1.7pt);
 \draw[black,fill=black!25] (5,6) circle (1.7pt);
 \draw[black,fill=black!25] (4,4) circle (1.7pt);
\filldraw[black!100] (7,5) circle (1.5pt);
 \filldraw[black!100] (8,6) circle (1.5pt);
   \filldraw[black!100] (7,4) circle (1.5pt);
 \filldraw[black!100] (8,3) circle (1.5pt);
 \draw[black,fill=black!25] (2,5) circle (1.7pt);
 \draw[black,fill=black!25] (2,6) circle (1.7pt);
 \draw[black,fill=black!25] (3,4) circle (1.7pt);
  \draw[black,fill=black!25] (4,3) circle (1.7pt);
 \draw[black,fill=black!25] (3,3) circle (1.7pt);
\end{tikzpicture}

\caption{Graph in which $4$-inseparable sets are not separated by optimally nested cuts}\label{fig:opt}
\end{figure}
\end{exam}
\

\begin{exam}\label{exam:four_ended}
In Figure~\ref {fig:4ends} an example is given of a $4$-ended graph in which $\mu_{\min}=4$. The vertices of four 3-element thin separators are drawn fat, corresponding sets $C\cup NC$ for cuts $C$ are shown in light grey. The best way to work out $\mu (C)$ for a particular cut $C$ is to count the number $s$ of thin separators
 that have points in both $C$ and $C^*$ and then $\mu(C) = 2s$.  For this graph any two rays that lie in distinct ends are separated by 
 a cut with $\mu (C) =\mu_{\min}= 4$.  The central block for the cut system of optimally nested cuts is shown in dark grey.  
 
 It is possible to change the graph of Figure~\ref{fig:opt} so that it gives an example of a graph in which there are ends that are separated
 by a thin cut but which are not separated by an optimally nested thin cut.
 \end{exam}

\vskip-3mm
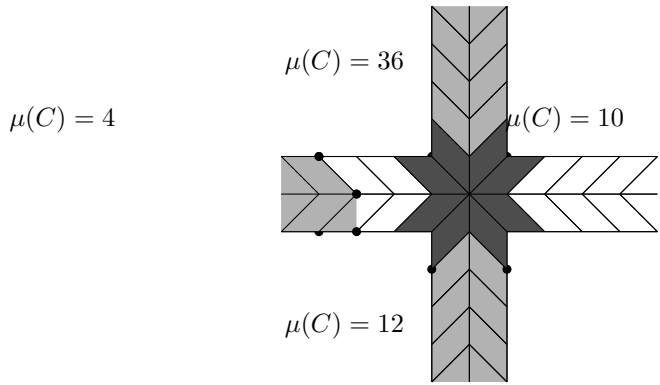
\begin{figure}[htbp]
\centering
\begin{tikzpicture}[scale=0.5]
\put (-90,0) {
\fill[black!30] (0,4) --(0,6)--(1,6) --(2,5)--(1,4)--(0,4) ;
\fill[black!30] (4,0) --(4,3)--(5,5) --(6,3)--(6,0)--(4,0) ;
\fill[black!30] (4,10) --(4,6)--(5,5) --(6,6)--(6,10)--(4,10) ;

\draw (0,5) -- (10,5) ;
\draw (5,10) -- (5,0) ;
\draw (0,6) -- (4,6) ;
\draw (6,6) -- (10,6) ;
\draw (4,10) -- (4,6) ;
\draw (4,0) -- (4,4) ;
\draw (0,4) -- (4,4) ;
\draw (6,4) -- (10,4) ;
\draw (6,10) -- (6,6) ;
\draw (6,0) -- (6,4) ;
\draw (4,4) --(6,6) ;
\draw (6,4)--(4,6);
\draw (3,6)--(4,5) --(3,4) ;
\draw (2,6)--(3,5) --(2,4) ;
\draw (1,6)--(2,5) --(1,4) ;
\draw (0,6)--(1,5) --(0,4) ;
\draw (10,6)--(9,5) --(10,4) ;
\draw (9,6)--(8,5) --(9,4) ;
\draw (8,6)--(7,5) --(8,4) ;
\draw (7,6)--(6,5) --(7,4) ;
\draw (4,10)--(5,9) --(6,10) ;
\draw (4,9)--(5,8) --(6,9) ;
\draw (4,8)--(5,7) --(6,8) ;
\draw (4,7)--(5,6) --(6,7) ;
\draw (4,0)--(5,1) --(6,0) ;
\draw (4,1)--(5,2) --(6,1) ;
\draw (4,2)--(5,3) --(6,2) ;
\draw (4,3)--(5,4) --(6,3) ;
 \filldraw (5,5) circle (3pt);
 \filldraw (6,3) circle (3pt);
 \filldraw (4,3) circle (3pt);
 \filldraw (2,5) circle (3pt);
 \filldraw (1,6) circle (3pt);
 \filldraw (1,4) circle (3pt);
 \filldraw (5,5) circle (3pt);
 \filldraw (4,6) circle (3pt);
 \filldraw (6,6) circle (3pt);
\draw (0.5, 7) node {$\mu(C) = 4$}  ;

\draw (8,8.5) node {$\mu(C) = 36$}  ;
\draw (8,1.5) node {$\mu(C) = 12$}  ;}
\put (100,0){
  \fill[black!70] (3,6)--(4,5)--(3,4)--(4,4)--(4,3)--(5,4)--(6,3)--(6,4)--(7,4)--(6,5)--(7,6)--(6,6)--(6,7)--(5,6)--(4,7)--(4,6)--(3,6) ;
\fill[black!30] (0,4) --(0,6)--(1,6) --(2,5)--(2,4)--(0,4) ;

\draw (0,5) -- (10,5) ;
\draw (5,10) -- (5,0) ;
\draw (0,6) -- (4,6) ;
\draw (6,6) -- (10,6) ;
\draw (4,10) -- (4,6) ;
\draw (4,0) -- (4,4) ;
\draw (0,4) -- (4,4) ;
\draw (6,4) -- (10,4) ;
\draw (6,10) -- (6,6) ;
\draw (6,0) -- (6,4) ;
\draw (4,4) --(6,6) ;
\draw (6,4)--(4,6);
\draw (3,6)--(4,5) --(3,4) ;
\draw (2,6)--(3,5) --(2,4) ;
\draw (1,6)--(2,5) --(1,4) ;
\draw (0,6)--(1,5) --(0,4) ;
\draw (10,6)--(9,5) --(10,4) ;
\draw (9,6)--(8,5) --(9,4) ;
\draw (8,6)--(7,5) --(8,4) ;
\draw (7,6)--(6,5) --(7,4) ;
\draw (4,10)--(5,9) --(6,10) ;
\draw (4,9)--(5,8) --(6,9) ;
\draw (4,8)--(5,7) --(6,8) ;
\draw (4,7)--(5,6) --(6,7) ;
\draw (4,0)--(5,1) --(6,0) ;
\draw (4,1)--(5,2) --(6,1) ;
\draw (4,2)--(5,3) --(6,2) ;
\draw (4,3)--(5,4) --(6,3) ;
  \filldraw (2,5) circle (3pt);
 \filldraw(1,6) circle (3pt);
 \filldraw(2,4) circle (3pt);
\draw (0.5, 7) node {$\mu(C) = 10$}  ;
}
   \end{tikzpicture}
\caption{Graph with 4 ends with $\mu_{\min}\not=0$}\label{fig:4ends}
\end{figure}\vskip-4mm\mbox{}
 
\begin{theo}\label{strong}    Let $X$ be a connected graph with automorphism group $G$, and let $\cc $ be an $\Omega $-cut system. There is a uniquely determined nested cut system $\cn$ invariant under $G$ which
is a subsystem of $\cc$ with the following property.
 If  $\omega _1, \omega _2 \in \Omega $ are separated by a cut in $\cc$ then they are separated by a 
cut in $\cn$.
\end{theo}

    \begin {proof}
As in Section 5, for $C \in \cc$ let $\mu (C)$ be the number of cuts in $\cc$ with which $C$ is not nested.

\begin {defi} \label {opti-wrt}Let $\cc (\omega _1, \omega _2)$ be the set of cuts in $\cc $ which separate $\omega _1, \omega _2$.  We say that $C \in \cc (\omega _1, \omega _2)$ is optimally nested
with respect to $\omega _1, \omega _2$ if $\mu (C)$ takes the smallest value among the elements of $\cc (\omega _1, \omega _2)$.
\end {defi}

Let $\cn $ be the set of cuts $C$  that are optimally nested with respect to some $\omega _1, \omega _2$.

We show that $\cn $ is a nested cut system.  This is an argument from \cite {DW}.   Suppose $C$ is optimally nested  with respect
to  $\omega _1, \omega _2$ and $D$ is optimally nested with respect to $\omega _3, \omega _4$.  Suppose $C, D$ are not nested.   Each $\omega _i$ determines a corner  $A_i$ of
$C, D$.   There are two possibilities.
\begin {itemize}
\item [(i)]  The sets $\omega _1, \omega _2$ determine opposite corners, and $\omega _3, \omega _4$ determine the other two corners.
\item [(ii)]  There is a pair  of  opposite corners such that one corner is determined by one of $\omega _1, \omega _2$ and the opposite corner
is determined by one of $\omega _3, \omega _4$.  
\end {itemize}

In case (i) $C$ and $D$ separate both pairs $\omega _1, \omega _2$ and $\omega _3, \omega _4$.   Since $C, D$ are optimally nested with respect to
$\omega _1, \omega _2$ and $\omega _3, \omega _4$,  we have   $\mu (C) = \mu(D)$.   But now  $A_1, A_2$ are opposite corners, and so $\mu (A_1) +\mu (A_2) < \mu (C)+\mu (D) =2\mu (C)$, by Lemma \ref {lemm:corners_equality}.  Since both $A_1, A_2$ separate $\omega _1$ and $\omega _2$ we have a contradiction.

In case (ii) suppose these corners are $A_1 = C\cap D$ and $A_3 = C^*\cap D^*$, and that $\omega_1$ belongs to $A_1\cup NA_1$
and $ \omega _3$ belongs to   $A_3\cup NA_3$.    But then $A_1$ separates $\omega _1$ and $\omega _2$ and $A_3$ separates $\omega _3$ and $\omega _4$.    Since $C$ is optimally nested with respect to $\omega _1$ and $\omega _2$ we have $\mu (A_1) \geq \mu (C)$ and since $D$ is optimally nested with respect to 
$\omega _3$ and $\omega _4$ we have $\mu (A_3) \geq \mu (D)$.   But it follows from  Lemma  \ref {lemm:corners_equality}
 that $\mu (A_1) + \mu(A_3) < \mu (C) + \mu(D)$ and so we have a
contradiction.    Thus $\cn$ is a nested cut system and the proof is complete.
\end{proof}

In the case of edge cuts,  one can take $\Omega $ to be the union of  set of edge ends and the set of vertices and obtain a sequence of
structure trees $T_k$ such that if two elements of $\Omega $ can be separated by removing $k$ edges then they are separated in the tree
$T_k$ (see \cite {D13} ).    It is not possible to get a similar result for vertex cuts as the following example shows.

\begin{exam}\label{exam:twoended}
Set $VX=\{v_i, u_j | i \in \Z, j \in \N\}$ and
\[EX=\{\{ v_i,v_{i+1} \} \mid i \in\Z\}\cup\{ \{ v_i, u_1\} \mid i\in\Z\}\cup\{\{u_j,u_{j+1} \} \mid j \in\N\}.\]
This graph is shown in Figure ~\ref{fig:9Exam}.

The following cut system $\cc$  is nested.     
The separators of $\cc$ are $\{ u_i\}$ for $i = 1, 2, \dots $ and $\{ u_1, v_i \}$ for $i \in \Z $.  We can construct structure trees for the set of $1$-separators (cut points) and also for the set of $2$-separators as shown, but there is no
natural way to construct a tree which includes information about all the separators.
The cut point tree $T(\cL)$ is as in the figure.   There is a block $B$ which consists of the full subgraph on
the vertices $\{ v_i \mid i \in \Z \}\cup \{ u_1 \}$.  This has a structure tree $T_B$ for $\kappa =2$ as also shown in the figure.

An automorphism of $X$ restricts to an automorphism of  $B$.
An automorphism of this  subgraph is induced by an automorphism of $X$ which fixes each $u_i$.
Thus the automorphism group of $X$ is an infinite dihedral group.

It is not possible to join the two trees  $T(\cL)$ and $T_B$  together in a way that admits the action of the automorphism group.
Thus the nested cut system $\cc $ does not correspond to a structure tree.

\end{exam}

\vskip-2mm
\begin{figure}[htbp]
\centering
\begin{tikzpicture}[scale=0.5]
    \path (0,0) coordinate (p0);     
        \path (1,0) coordinate (p7);  
         \path (8,0) coordinate (p8);           
    \path (2,0) coordinate (p1);
    \path (3,0) coordinate (p2);
    \path (4,0) coordinate (p3);
    \path (5,0) coordinate (p4);
    \path (6,0) coordinate (p5);
  \path (7,0) coordinate (p6);
\path (4,2) coordinate (q1);
    \path (4,4) coordinate (q2);
    \path (4,5) coordinate (q3);
    \path (4,5.6) coordinate (q4);
    \path (4,6) coordinate (q5);
    \draw (p1) -- (p5);
\draw [dashed] (p0) --(p1);
\draw [dashed] (p5) --(p8);
\draw [dashed] (p0)--(q1)--(p8);
\draw [dashed] (q1) -- (p7);
\draw [dashed] (q1) -- (p6);
  \filldraw (p7) circle (3pt) ;
\filldraw (p0) circle (3pt) ;
 \filldraw (p6) circle (3pt) ;
 \filldraw (p8) circle (3pt) ;
 
\filldraw (q3) circle (3pt) ;
 \filldraw (q4) circle (3pt) ;
 \filldraw (q5) circle (3pt) ;

  \filldraw (p1) circle (3pt) ;
  \filldraw (p2) circle (3pt);
  \filldraw (p3) circle (3pt);   
  \filldraw (p4) circle (3pt);
  \filldraw (p5) circle (3pt) ;
  \filldraw (q1) circle (3pt);  
 \filldraw (q2) circle (3pt);
\draw (p1) -- (q1) -- (p2);
     \draw (p3) -- (q1) -- (p4);
 \draw (p5) -- (q1) ;
\draw (q1) -- (q3) ;
\draw [dashed] (q3)--(4,6.6) ;
\path (10,2) coordinate (q1);
    \path (10,4) coordinate (q2);
    \path (10,5) coordinate (q3);
    \path (10,5.6) coordinate (q4);
\filldraw (q3) circle (3pt) ;
 \filldraw (q4) circle (2pt) ;
   \filldraw (q1) circle (3pt);  
 \filldraw (q2) circle (3pt);
\draw (q1) -- (q3) ;
\draw [dashed] (q3)--(10,6.6) ;
\path (10,3) coordinate (q1);
    \path (10,4.5) coordinate (q2);
    \path (10,5.3) coordinate (q3);
  \filldraw (q1) circle (3pt);  
 \filldraw (q2) circle (3pt);
  \filldraw (q3) circle (2pt);  
 \filldraw [white] (q2) circle (2pt);
  \filldraw [white] (q1) circle (2pt);  
 \filldraw[white] (q3) circle (1pt);

      \path (13,0) coordinate (p1);
    \path (14,0) coordinate (p2);
    \path (15,0) coordinate (p3);
    \path (16,0) coordinate (p4);
   \path (16.5,0) coordinate (p5);
   \path (10,0) coordinate (p6);
\path (12.5, 0) coordinate (r);
     \draw (p1) -- (p4);

     \path (13.5,0) coordinate (q1);
    \path (14.5,0) coordinate (q2);
    \path (15.5,0) coordinate (q3);
    \filldraw (p1) circle (3pt) ;
   
     \filldraw (p2) circle (3pt);
  \filldraw (p3) circle (3pt);  
   \filldraw (q1) circle (3pt) ;
   
     \filldraw (q2) circle (3pt);
  \filldraw (q3) circle (3pt);  
 \filldraw (q1) [white]  circle (2pt) ;
   
     \filldraw (q2) [white] circle (2pt);
  \filldraw (q3) [white] circle (2pt);  
\draw [dashed] (p1) -- (12, 0);
 \draw [dashed] (p4) -- (17, 0);
  \filldraw (p4) circle (3pt);
  \filldraw (p6) [white] circle (2pt);
\path (5.8, 2) coordinate (q);
    \path (5.8, 4) coordinate (r);
\draw node at (3.5,2.4) {$u_1$};
\draw node at (3.5,4) {$u_2$};
   
   \draw (4.5,-1) coordinate (p1);
\draw (p1) node [left ] {$X$};

  \draw (10, -1) coordinate (p1);
\draw (p1) node {$T(\cL )$ };
  \draw (14.5, -1) coordinate (p1);
\draw (p1) node {$T_B$ };
 \draw (10, 2) coordinate (p1);
\draw [left] (p1) node {$B$ };

\end{tikzpicture}
\vskip2mm\caption{Structure trees for a three-ended graph}\label{fig:9Exam}
\end{figure}
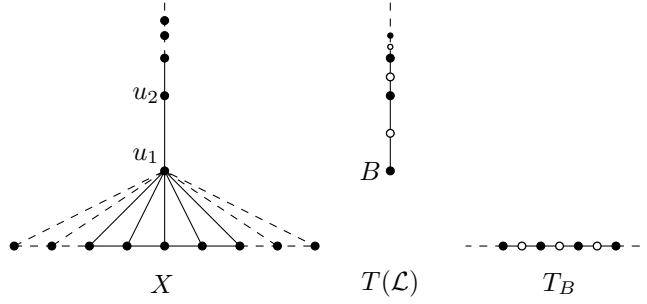

In the case when $\Omega $ is the set of $k$-inseparable sets the following result gives more information about a block $B$
and the graph $X_B$.
 \begin {theo}  Let $\Omega $ be the set of $k$-inseparable sets and  let $T = T(\cn )$  be the structure tree corresponding to the nested sub-system $\cn $ of the $\Omega $-cut system $\cc $.
 
 Each block  $B$ contains at most one maximal  $\kappa $-inseparable set.
  If a  subset of $B$  is $k$-inseparable in $X$ then it is $k$-inseparable in $X_B$.
  
 A subset of $B$ which is $k$-inseparable in   $X_B$ is $k$-inseparable in $X$, unless there is 
a cut $C$ in $\cn $ such that 
\begin {itemize} 
\item [(i)]    $NC \subset B$,
 and
 \item [(ii)]  $C\cup NC$ is a block $B'$ containing a  maximal $\kappa $-inseparable set, and $|B'| \leq 3\kappa /2$,
 \end {itemize}  
 \end {theo}
 \begin {proof}  Here  $\cn$ is the uniquely determined nested sub-system of Theorem ~\ref {strong}. 
 A $k$-inseparable set is $\kappa $-inseparable and so must be a subset of a block of $T$.
Two distinct maximal $\kappa $-inseparable sets can be separated by a $\kappa $-separator and so must lie in distinct blocks.
 
 If $S$ is a $k$-separator in $X_B$ then it is also a $k$-separator in $X$.   For if $A \subset B$ and $|NA| = k$ where $NA$ is the set of vertices of
 $B \setminus A$ which are adjacent to a vertex of $A$ in $X_B$, then  for every separator $NC$   of $\cn$ that is adjacent to $B$ we have
 either $NC   \subset A\cup NA$ or $NC \subset A^* \cup NA$, since $NC$ is the vertex set of a complete subgraph of $X_B$ and so no pair of vertices of $NC$ can be separated by $NA$.   Thus $NA = NA_X$, where $A'$ is the subset of $VX$ consisting of the union of $A$ with  cuts $C$ and slices of $\cn $ that are disjoint from $B$ but which have $NC \subset A\cup NA$.    
 
  Conversely  if $A_X$ is connected in $X$ and $NA_X \subset B$,
 then for every cut $C$ in $\cn $  such  that is $NC$ adjacent to $B$ we have either $C\subset A_X$ or $C\subset A_X^*$ and so $NC \subset A_X\cup NA_X$
 or $NC \subset A_X^*\cup NA_X$.   Thus $A  = A_X \cap B$ gives a separation of $X_B$ in which $NA = NA_X$.
 
 Finally we will show that if  if two $k$-inseparable sets  $\omega _1, \omega _2$  are contained in $B$ and are separated by a cut   $D$  in $X$, then they are separated
 by a connected cut, with the additional property that  $ND \subset B$.
We need to show that we can assume that $D$ is nested with every $C\in \cn$.    Here we will use the fact that no cut $C \in \cn$ satisfies both  conditions (i) and (ii).
     
 Suppose $C$ separates $\omega _1, \omega _2  \in \Omega$ where $C \in \cn$ and so  $|NC| = \kappa $.
  Similarly suppose $D$ separates $\omega _3, \omega _4$ , where $\omega _3$ and $\omega _4$ are $k$-inseparable sets with  $k  >\kappa $. and $|ND| =k $ is minimal for the cuts separating $\omega _3, \omega _4$.   Suppose that $C, D$ are not nested.
 We refer to Figure~\ref{fig:cross}.    We see that $|N(C\cap D)| + |N(C^*\cap D^*)| \leq |NC| + |ND|$ with equality if and only if every vertex of $NC\cap ND$ is
 in both $N(C\cap D)$ and $N(C^*\cap D^*)$,
and that  $|N(C\cap D^*)| + |N(C^*\cap D)| \leq |NC| + |ND|$ also with equality if and only if $NC\cap ND = N(C\cap D^*)= N(C^*\cap D)$.

We want to show that there exists a corner, say $C^* \cap D$,  that separates  $\omega _3, \omega _4$, with $|N(C^*\cap  D)| = k$.

Each  of  $\omega _3, \omega _4$ determines a unique corner $A_3, A_4$ of $C$ and $D$ such that $\omega_i$ is contained in the union of $A_i$, the two links which are adjacent to $A_i$ and the center.  Even though $A_i$ is uniquely determined, we cannot rule out the possibility
that $\omega _i \cap A_i=\emptyset$ at this stage.     If $\omega _3, \omega _4$  determine opposite corners - say $C\cap D, C^*\cap D^*$ - then we get a contradiction
since $|N(C\cap D)| + |N(C^*\cap D^*)| \leq k +\kappa < 2k$ and each of $N(C\cap D), N(C^*\cap D^*)$ separates $\omega _3, \omega _4$, which are $k$-inseparable, but cannot be separated by less than $k$ vertices.
Thus $\omega _3, \omega _4$ determine  adjacent corners, which we take to be $C^*\cap D$ and $C^*\cap D^*$.     It now follows that $b \geq a$,  $d\geq a$, since
otherwise one of $|N(C^*\cap D)| < k$, $|N(C^*\cap D^)| < k$.    In fact if $a< b\leq d$, then $|N(C\cap D^*)| < \kappa $ and $|N(C\cap D)| < \kappa $.  But $C\cup NC$ contains a $\kappa $-inseparable set $Y_1$. If $\omega _1$ has empty intersection with both $C\cap D$ and $C\cap D^*$, then $|\omega _1|\leq a + b +m +d <3\kappa /2$, since $a < \kappa /2$ and $b+m+d = \kappa $.  Since   (ii) does not hold we see that $\omega _1$ has non-empty intersection with
one of the two corners.   We get a contradiction since $\omega _1, \omega _2$ are separated by a set with less than $\kappa $ elements.  Thus $a = b$.     It follows that  $|N(C^*\cap D^*)| = k$ and $C^*\cap D^*$ separates $\omega _3$ and $\omega _4$.

If $\mu (D)$ is the number of cuts of $\cn$ with which $D$ is not nested, then $\mu (C^*\cap D^*) < \mu (D)$.   Repeating this process,
we eventually obtain a cut $D$ for which $\mu (D) =0$ and hence $ND$ is a subset of $B$.   
  
  \end {proof}
  It follows from the last theorem that if we are given a graph $X$ in which two $k$-inseparable sets  $\omega _1, \omega _2$  can be separated by
  removing $k$ vertices then we can form the unique structure tree $T = T(\cn)$.   Our two sets either belong to the same block $B$ or they are separated
  in $T$.   If they are not separated in $T$ then we can consider the graph  $X_B$.    Let $\kappa _B$ be the smallest value of $k$ for which
  there are $k$-inseparable subsets of $B$ that can be separated by removing $k$ vertices.  Then $\kappa _B > \kappa $.
  We can then construct a structure tree for $X_B$ and repeat the process.    We can keep repeating this process until we separate $\omega _1$ and $\omega _2$, unless in this process there is a vertex of $T$ that has degree one and which corresponds to a cut $C$  satisfying (i) and (ii).     In this latter case it may happen that $\omega _1,  \omega _2$ can be separated in $X$ but cannot be separated in $X_B$  by removing $k$ vertices.  If this happens, then $\omega _1 \cup \omega _2$ will be  $k$-inseparable in $X_B$.
  
  In the case when $\Omega $ is the set of ends of $X$ we can obtain the following in a similar way.
  \begin {theo}   Let $T = T(\cn )$  be the structure tree corresponding to the $\Omega $-cut system $\cc $.
 
 An end of $T$ corresponds to an element of $\Omega $.   An element of $\Omega $ that does not correspond to an end of $T$ belongs to a unique
 block $B$.
  Two elements of $\Omega $ that belong to  $B$  can be separated by $k$ vertices in $X$  if and only if  they can be separated by
  $k$ vertices in $X_B$.

 \end {theo}
 
 We can carry out the process as described above for when $\Omega $ is the set of $k$-inseparable sets.   In this case, when $\Omega $ is the
 set of ends, for any pair of ends $\omega _1,\omega _2$ we can repeat the process until we separate the ends.    Thus we will eventually obtain a block $B$ such that $\omega _1, \omega _2$ belong to $B$ and in the tree $T_B$ they either belong to distinct ends or to distinct block vertices
 or to an end and a vertex.

\section{Applications}
\subsection{Structure trees for finite graphs.}\label{subsect:structure_trees_finite}   Our main result for finite graphs is summarised as follows.

{\it Let $X$ be a finite graph with automorphism group $G$.
By Theorem \ref {strong},  there is a canonically determined thin  cut system $\cn $ invariant under $G$ which separates any pair of maximal $\kappa $-inseparable sets.  This gives rise to a canonical structure tree $T = T(\cn )$. This tree will admit an action of $G$.}

We illustrate the theory  in some examples.   

From Theorem~\ref{theo:mu=0} we know that $\mu_{\min} = 0$.   The set of cuts $C$ with $\mu (C) = 0$ form a nested subsystem $\cn $ and in
these  examples each  maximal $\kappa $-inseparable set is an $\cn $-block.  This is not  always be the case, as can be seen from Figure~\ref{fig:opt}.
We now give some examples.
\begin{exam}
The following example illustrates a tree decomposition of a 3-connected graph.
We take $X$ to be the graph shown in Figure~\ref{fig:3block} (a) and (b). Let $\cn$ be the  sub-system of  cuts nested with all cuts  in the cut system $\ci $ from Example~\ref{exam:finitecuts}.  In fact in this case $\cn = \ci$.
We have  $\kappa = 3$. There are eight ($3$-inseparable) blocks. Three of them are shown shaded dark in (a).   
The three corresponding vertices of the structure tree all have valency one.
Four blocks each consists of a $K_3$ or a $K_4$, shaded light in (a), together with $p$. 
There is another block which has valency $3$ in the structure tree. This is shown in (b) with dotted ideal edges.
The structure tree is shown in (c).

\begin{figure}[htbp]
\centering
\begin{tikzpicture}[scale=0.5]
\put(-120,18.2){
\path (0,2) coordinate (p1);
    \path (2,2) coordinate (p2);
    \path (0,0) coordinate (p3);
    \path (2,0) coordinate (p4);
    \path (4,0) coordinate (q2);
    \path (3,-3.2) coordinate (q3);
    \path (4.8, -2) coordinate (q4);
    \path (1.2, -2) coordinate (q5);
      \path (6,2) coordinate (r1);
    \path (4,2) coordinate (r2);
    \path (4,0) coordinate (r3);
    \path (6,0) coordinate (r4);
    \path (2, 4) coordinate (s3);
    \path (4,4) coordinate (s4);
    \path (6,4) coordinate (t3);    
    \path (3,3) coordinate (M);    
    \path (5,1) coordinate (N);    
     \path (intersection of p4--q3 and q2--q5) coordinate (Y);
      \path (intersection of p4--q4 and q2--q5) coordinate (Z);
 \path (intersection of p1--p4 and p2--p3) coordinate (X);
 \path (intersection of r1--s4 and r2--t3) coordinate (W);
 \path (intersection of p4--q4 and q2--q3) coordinate (V);
    
   \fill[black!15] (p1)--(r2)--(s4)--(s3)--(p1);
   % \fill[black!15] (p2)--(s4)--(r2);
           \fill[black!15] (r1)--(r2)--(q2)--(q4)--(r4)--(r1);
     %  \fill[black!15] (r1) --(r2)--(r3);
            \fill[black!45] (q2) --(q4)--(q3)--(q5)--(p4);
  \fill[black!45] (p2)--(p4)--(p3)--(p1);
 \fill[black!45] (r1)--(r2)--(s4)--(t3);
    
    \draw (p1)--(p4);
    \draw (q3)--(p4)--(q4);
    \draw (s4)--(r1);
    
    \draw (s3)--(r2);
    \draw (r2)--(r4);
    
 \filldraw[black!45] (X) circle (3pt);
 \filldraw[black!15] (M) circle (3pt);
 \filldraw[black!15] (N) circle (3pt);
 \filldraw[black!45] (Y) circle (3pt);
 \filldraw[black!45] (Z) circle (3pt);
 \filldraw[black!45] (W) circle (3pt);
 \filldraw[black!45] (V) circle (3pt);    
    \draw (2.3,0.4) node {$p$};
    
\draw (p4)--(s3)--(t3)--(r1)--(p1)--(p3)--(p4);
\draw (p3)--(s4)--(r2);
\draw (t3)--(r2)--(q2)--(r1);
 \draw(p1)--(s3);
 \draw(p4)--(q5)--(q3)--(q4)--(q2)--(r4)--(r1);
 \draw (r4)--(q4);
 \draw (q5)--(q2)--(q3);
    \filldraw (p1) circle (2pt) ;
    \filldraw (p2) circle (2pt) ; 
    \filldraw (p3) circle (2pt) ;
    \draw [dashed] (p4)-- (q2);
    \filldraw (p4) circle (2pt) ;
    \filldraw (q2) circle (2pt) ; 
    \filldraw (q3) circle (2pt) ;
      \filldraw (q4) circle (2pt) ;
 \filldraw (q5) circle (2pt) ;  
 \filldraw (r4) circle (2pt) ;
   \filldraw (r1) circle (2pt) ;    
     \filldraw (r2) circle (2pt) ;
    \filldraw (s4) circle (2pt) ;
  \filldraw (s3) circle (2pt) ;
 \filldraw (t3) circle (2pt) ;
 \draw (3,-5.5) node {a};
}

 \put(20,18.2){
 
\path (0,2) coordinate (p1);
    \path (2,2) coordinate (p2);
    \path (0,0) coordinate (p3);
    \path (2,0) coordinate (p4);
    \path (4,0) coordinate (q2);
    \path (3,-3.2) coordinate (q3);
    \path (4.8, -2) coordinate (q4);
    \path (1.2, -2) coordinate (q5);
      \path (6,2) coordinate (r1);
    \path (4,2) coordinate (r2);
    \path (4,0) coordinate (r3);
    \path (6,0) coordinate (r4);
    \path (2, 4) coordinate (s3);
    \path (4,4) coordinate (s4);
    \path (6,4) coordinate (t3);    
     \path (intersection of p4--q3 and q2--q5) coordinate (Y);
      \path (intersection of p4--q4 and q2--q5) coordinate (Z);
 \path (intersection of p1--p4 and p2--p3) coordinate (X);
 \path (intersection of r1--s4 and r2--t3) coordinate (W);
 \path (intersection of p4--q4 and q2--q3) coordinate (V);    
 
    \draw (p1)--(p4);
    \draw (q3)--(p4)--(q4);
    \draw (s4)--(r1);
    
       \fill[black!15] (s4)--(p4)--(r1)--(r2)--(s4);
    
    \draw (s3)--(r2);
    \draw (r2)--(r4);
    
 \filldraw[white] (M) circle (3pt);
 \filldraw[white] (N) circle (3pt);
 \filldraw[white] (X) circle (3pt);
 \filldraw[white] (Y) circle (3pt);
 \filldraw[white] (Z) circle (3pt);
 \filldraw[white] (W) circle (3pt);
 \filldraw[white] (V) circle (3pt);    
 \fill[black!45] (r1)--(r2)--(s4);
 \draw[dashed] (s4)--(p4)--(r2);
 \draw[dashed] (p4)--(r1);
    \draw (1.5,-0.4) node {$p$};
\draw (p4)--(s3)--(t3)--(r1)--(p1)--(p3)--(p4);
\draw (p3)--(s4)--(r2);
\draw (t3)--(r2)--(q2)--(r1);
 \draw(p1)--(s3);
 \draw(p4)--(q5)--(q3)--(q4)--(q2)--(r4)--(r1);
 \draw (r4)--(q4);
 \draw (q5)--(q2)--(q3);
    \filldraw (p1) circle (2pt) ;
    \filldraw (p2) circle (2pt) ; 
    \filldraw (p3) circle (2pt) ;
    \draw [dashed] (p4)-- (q2);
    \filldraw (p4) circle (2pt) ;
    \filldraw (q2) circle (2pt) ; 
    \filldraw (q3) circle (2pt) ;
      \filldraw (q4) circle (2pt) ;
 \filldraw (q5) circle (2pt) ;  
 \filldraw (r4) circle (2pt) ;
   \filldraw (r1) circle (2pt) ;    
     \filldraw (r2) circle (2pt) ;
    \filldraw (s4) circle (2pt) ;
  \filldraw (s3) circle (2pt) ;
 \filldraw (t3) circle (2pt) ;
\draw (3,-5.5) node {b};
 }

 \put(170,0){    
 \path (0,7.2) coordinate (p1);
    \path (0,6.6) coordinate (p2);
    \path (0,6) coordinate (p3);
  \draw (0,5.4) coordinate (p4);
   \draw (0,4.8) coordinate (p5);
    \draw (0,4.2) coordinate (p6);
       \draw (0,3.6) coordinate (p7);
    \draw (0,3) coordinate (p8);
    \draw (0, 2.4) coordinate (p9);
   \draw (0, 1.8) coordinate (p10);
   \draw (0, 1.2) coordinate (p11);
   \draw (0, 0.6) coordinate (p12);
   \draw (0, 0) coordinate (p13);
   \draw (0, -0.6) coordinate (p14);
   \draw (0, -1.2) coordinate (p15);
   \draw (0, -1.8) coordinate (p16);
   \draw (0, -2.4) coordinate (p17);
   \draw (0.6, 2.4) coordinate (p18);
   \draw (1.2, 2.4) coordinate (p19);
  %   \filldraw (p1) circle (3pt) ;
%\filldraw (p2) circle (3pt);
   \filldraw (p3) circle (3pt);  
 \filldraw (p4) circle (3pt);
\filldraw (p5) circle (3pt) ;
 \filldraw (p6) circle (3pt);  
\draw (p3) --(p15);
 \draw (p9) -- (p19);
 
%\filldraw [white] (p2) circle (3pt);
%\draw (p2) circle (3pt);
\filldraw [white] (p4) circle (3pt);
\draw (p4) circle (3pt);
\filldraw [white] (p6) circle (3pt);
\draw (p6) circle (3pt);
\filldraw [white] (p8) circle (3pt);
\draw (p8) circle (3pt);
\filldraw [white] (p10) circle (3pt);
\draw (p10) circle (3pt);
\filldraw [white] (p12) circle (3pt);
\draw (p12) circle (3pt);
\filldraw [white] (p14) circle (3pt);
\draw (p14) circle (3pt);
%\filldraw [white] (p16) circle (3pt);
%\draw (p16) circle (3pt);
\filldraw [white] (p18) circle (3pt);
\draw (p18) circle (3pt);
 
\filldraw (p11) circle (3pt);
%\filldraw (p1) circle (3pt);
\filldraw (p3) circle (3pt);
\filldraw (p5) circle (3pt);
\filldraw (p7) circle (3pt);
\filldraw (p9) circle (3pt);
\filldraw (p11) circle (3pt);
\filldraw (p13) circle (3pt);
\filldraw (p15) circle (3pt);
%\filldraw (p17) circle (3pt);
\filldraw (p19) circle (3pt);
\draw (0.1,-4.2) node {c};
}

\end{tikzpicture}
\vskip-0.6cm \caption{Decomposition of a 3-connected graph}\label{fig:3block}\vskip-3mm
\end{figure}
\end{exam}%\vskip-12mm\mbox{}

\begin{exam}
In the next example there is an $\cn$-block which is also an $\ci$-block and which is not connected. 
It consists of the three vertices on the top together with the three vertices at the bottom. In the tree it corresponds to the central vertex. The three ``vertical sides'' of the graph, which are unions of three tines of the dragon's neck, correspond to a $\cn$-block which is not an $\ci$-block.

\vskip 5mm
\begin{figure}[htbp]
\centering
\begin{tikzpicture}[scale=0.45]

\put(-80,0){
\path (0,1) coordinate (a1); 
\path (0,3) coordinate (a2); 
\path (0,5) coordinate (a3); 
\path (0,7) coordinate (a4); 
\path (3,0) coordinate (b1); 
\path (3,2) coordinate (b2); 
\path (3,4) coordinate (b3); 
\path (3,6) coordinate (b4); 
\path (4,3) coordinate (c1); 
\path (4,5) coordinate (c2); 
\path (4,7) coordinate (c3); 
\path (4,9) coordinate (c4); 
\path (-2.4,1.7) coordinate (A1); 
\path (-2.4,3.7) coordinate (A2); 
\path (-2.4,5.7) coordinate (A3); 
\path (4.8,-0.9) coordinate (B1); 
\path (4.8,1.1) coordinate (B2); 
\path (4.8,3.1) coordinate (B3); 
\path (5.2,5.9) coordinate (C1); 
\path (5.2,7.9) coordinate (C2); 
\path (5.2,9.9) coordinate (C3); 
\draw (c4)--(a4)--(b4)--(c4)--(C3)--(c3)--(C2)--(c2)--(C1)--(c1)--(c4);
\draw (a1)--(c1)--(b1);
\filldraw[white] (intersection of c1--C1 and b4--B3) circle (4pt);
\filldraw[white] (intersection of c1--c4 and b4--B3) circle (4pt);
\filldraw[white] (intersection of c1--c4 and b3--B3) circle (4pt);
\filldraw[white] (intersection of c1--a1 and b1--b4) circle (4pt);
\filldraw[white] (intersection of c1--a1 and b3--B2) circle (4pt);
\filldraw[white] (intersection of c1--b1 and B1--b2) circle (4pt);
\filldraw[white] (intersection of c1--b1 and b2--B2) circle (4pt);
\filldraw[white] (intersection of c1--b1 and b3--B2) circle (4pt);
\draw (b1)--(b4)--(B3)--(b3)--(B2)--(b2)--(B1)--(b1)--(a1)--(a4)--(A3)--(a3)--(A2)--(a2)--(A1)--(a1);
\filldraw (a1) circle (2pt);
\filldraw (a2) circle (2pt);
\filldraw (a3) circle (2pt);
\filldraw (a4) circle (2pt);
\filldraw (A1) circle (2pt);
\filldraw (A2) circle (2pt);
\filldraw (A3) circle (2pt);
\filldraw (b1) circle (2pt);
\filldraw (b2) circle (2pt);
\filldraw (b3) circle (2pt);
\filldraw (b4) circle (2pt);
\filldraw (B1) circle (2pt);
\filldraw (B2) circle (2pt);
\filldraw (B3) circle (2pt);
\filldraw (c1) circle (2pt);
\filldraw (c2) circle (2pt);
\filldraw (c3) circle (2pt);
\filldraw (c4) circle (2pt);
\filldraw (C1) circle (2pt);
\filldraw (C2) circle (2pt);
\filldraw (C3) circle (2pt);
}

\put(80,40){
\draw (330:2) +(330:2)--(0,0);
\draw  (330:2)+(260:2) --(330:2);
\draw (330:2)+(40:2)--(330:2);

\draw (210:2) +(210:2)--(0,0);
\draw (210:2)+(280:2)--(210:2);
\draw (210:2)+(140:2)--(210:2);

\draw (0,4)--(0,0);
\draw (0,2)+(20:2)--(0,2);
\draw (0,2)+(160:2)--(0,2);
\filldraw(0:0) +(0,2) circle (3pt);
\draw(20:1) [fill=white] +(0,2) circle (3pt);
\draw(90:1) [fill=white] +(0,2) circle (3pt);
\draw(160:1) [fill=white] +(0,2) circle (3pt);
\draw(270:1) [fill=white] +(0,2) circle (3pt);
\filldraw(20:2) +(0,2) circle (3pt);
\filldraw(90:2) +(0,2) circle (3pt);
\filldraw(160:2) +(0,2) circle (3pt);
\filldraw(270:2) +(0,2) circle (3pt);
\filldraw(0:0) +(210:2) circle (3pt);
\draw(30:1) +(210:2) [fill=white] circle (3pt);
\draw(140:1) +(210:2) [fill=white] circle (3pt);
\draw(210:1) +(210:2) [fill=white] circle (3pt);
\draw(280:1) +(210:2) [fill=white] circle (3pt);
\filldraw(30:2) +(210:2) circle (3pt);
\filldraw(140:2) +(210:2) circle (3pt);
\filldraw(210:2) +(210:2) circle (3pt);
\filldraw(280:2) +(210:2) circle (3pt);
\filldraw(0:0) +(330:2) circle (3pt);
\draw(40:1) +(330:2) [fill=white] circle (3pt);
\draw(150:1) +(330:2) [fill=white] circle (3pt);
\draw(260:1) +(330:2) [fill=white] circle (3pt);
\draw(330:1) +(330:2) [fill=white] circle (3pt);
\filldraw(40:2) +(330:2) circle (3pt);
\filldraw(150:2) +(330:2) circle (3pt);
\filldraw(260:2) +(330:2) circle (3pt);
\filldraw(330:2) +(330:2) circle (3pt);
}
\end{tikzpicture}
\vskip-9mm\caption{Dragon's Neck Graph}\label{fig:dragon}\vskip-2mm
\end{figure}

\vskip-8mm
\end{exam}

\begin{exam}
The third example is  the graph in Figure~\ref{fig:Tutte} which is similar to the example of Figure IV.3.4 from Tutte's book \cite{Tutte1984}. In contrast to Tutte, we are not considering multiple edges, because multiple and single edges are indistinguishable for vertex cuts.

\begin{figure}[htbp]
\centering
\begin{tikzpicture}[scale=0.4]
\put(-91,0){
\path (3,8.7) coordinate (a1); 
\path (3,7.3) coordinate (a2); 
\path (5,8) coordinate (p1); 
\path (6,11) coordinate (p2);  
\path (3,12) coordinate (p3); 
\path (3,10) coordinate (p4);  
\path (0,11) coordinate (p5);  
\path (1,8) coordinate (p6);  
\path (3,6) coordinate (p7);  
\path (3,4) coordinate (p8);  
\path (1,4) coordinate (p9);  
\path (1,2) coordinate (p10);  
\path (3,2) coordinate (p11);  
\path (5,4) coordinate (p12);   
\path (5,2) coordinate (p13);   
\path (5,0) coordinate (p14);   
\path (7,2) coordinate (p15);   
\path (7,4) coordinate (p16);   
\path (9,6) coordinate (p17);  
\path (9.27,6.09) coordinate (nn);  
\path (8,9) coordinate (p18);   
\path (11,10) coordinate (p19);  
\path (12,7) coordinate (p20); 
\path (7,7) coordinate (p21); 
\path (9,11) coordinate (p22); 
\path (13,9) coordinate (p23); 
\path (10.8,5.7) coordinate (p24);
\path (11.5,3.5) coordinate (p25);
\filldraw[rounded corners=4pt,black!15]  (2.6,5)--(2.6,6.5)--(3.4,6.5)--(3.4,3.5)--(2.6,3.5)--(2.6,5);
\filldraw[rounded corners=4pt,black!15]  (4.6,2)--(4.6,4.4)--(7.4,4.4)--(7.4,1.7)--(5.3,-0.4)--(4.6,-0.4)--(4.6,2);
\filldraw (a1) [black!15] circle (12pt);
\filldraw (a1) circle (3pt);
\filldraw (a2) [black!15] circle (12pt);
\filldraw (a2) circle (3pt);
\filldraw (p2) [black!15] circle (12pt);
\filldraw (p2) circle (3pt);
\filldraw (p3) circle (3pt);
\filldraw (p4) circle (3pt);
\filldraw (p5) [black!15] circle (12pt);
\filldraw (p5) circle (3pt);
\filldraw (p7) circle (3pt);
\filldraw (p8) circle (3pt);
\filldraw (p9) circle (3pt);
\filldraw (p10) [black!15] circle (12pt);
\filldraw (p10) circle (3pt);
\filldraw (p11) circle (3pt);
\filldraw (p12) circle (3pt);
\filldraw (p13) circle (3pt);
\filldraw (p14) circle (3pt);
\filldraw (p15) circle (3pt);
\filldraw (p16) circle (3pt);
\filldraw (p18) circle (3pt);
\filldraw (p19)circle (3pt);
\filldraw (p21) circle (3pt);
\filldraw (p22) [black!15] circle (12pt);
\filldraw (p22) circle (3pt);
\filldraw (p23) [black!15] circle (12pt);
\filldraw (p23) circle (3pt);
\filldraw (p24) [black!15] circle (12pt);
\filldraw (p24) circle (3pt);
\filldraw (p25) [black!15] circle (12pt);
\filldraw (p25) circle (3pt);
\draw (p1)--(p6)--(p8)--(p1) -- (p2) -- (p3)--(p1)--(p4)--(p3) -- (p5)--(p6)--(p3);
\draw (p4)--(p6) -- (a1)-- (p1)--(p7)--(p1)--(a2)--(p6);
\draw (p8)--(p7) -- (p6) -- (p9)--(p10)--(p11)--(p9);
\draw (p11)--(p15)--(p12)--(p14)--(p11)--(p12);
\draw (p14)--(p15)--(p16)--(p12)--(p17)--(p16);
\draw (p15)--(p17)--(p18)--(p19)--(p20)--(p17);
\draw (p1)--(p17)--(p24)--(p20)--(p23)--(p22)--(p22)--(p21);
\draw (p17)--(p25)--(p20);
\draw (p1) [fill=white] circle (4pt);
\draw (p6) [fill=white] circle (4pt);
\draw (p17) [fill=white] circle (4pt);
\draw (p20) [fill=white] circle (4pt);
\draw (p1) node [below] {\hskip1mm 1};
\draw (p2) node [right] {\hskip0.6mm 2};
\draw (p3) node [above] {3};
\draw (p4) node [above] {\hskip2.55mm 4};
\draw (a1) node [above] {5};
\draw (a2) node [below] {6};
\draw (p7) node [below] {\hskip2.5mm 7};
\draw (p8) node [below] {8};
\draw (p5) node [left] {{}\hskip-5mm 9};
\draw (p6) node [left] {10};
\draw (p9) node [left] {11};
\draw (p10) node [left] {{}\hskip-7mm 12};
\draw (p11) node [below] {\hskip-2mm 13};
\draw (p12) node [above] {\hskip-2mm 14};
\draw (p13) node [above] {\hskip-5mm 15};
\draw (p14) node [below] {16};
\draw (p16) node [below] {\hskip-4.4mm 17};
\draw (p15) node [below] {\hskip2mm 18};
\draw (nn) node [above] {\hskip0.8mm 19};
\draw (p18) node [left] {20};
\draw (p19) node [above] {\hskip2mm 21};
\draw (p20) node [right] {22};
\draw (p21) node [below] {\hskip-1mm 24};
\draw (p22) node [above] {25};
\draw (p23) node [right] {26};
\draw (p24) node [below] {\hskip1mm 23};
\draw (p25) node [below] {27};
\draw (7.7,-1.5) node {13a};
}

\put(91,0){
\fill[black!35] (p1)--(p3)--(p6)--(p4);
\fill[black!35] (p17)--(p18)--(p21)--(p17);
\fill[black!10] (p1)--(p6)--(p9)--(p11)--(p17)--(p21)--(p1);
\fill[black!10] (p17)--(p18)--(p19)--(p20)--(p17);
\draw [dashed] (p17)--(p11) ;
\filldraw (a1) circle (3pt);
\filldraw (a2) circle (3pt);
\filldraw (p1) circle (3pt);
\filldraw (p2) circle (3pt);
\filldraw (p3) circle (3pt);
\filldraw (p4) circle (3pt);
\filldraw (p5) circle (3pt);
\filldraw (p6) circle (3pt);
\filldraw (p7) circle (3pt);
\filldraw (p8) circle (3pt);
\filldraw (p9) circle (3pt);
\filldraw (p10) circle (3pt);
\filldraw (p11) circle (3pt);
\filldraw (p12) circle (3pt);
\filldraw (p13) circle (3pt);
\filldraw (p14) circle (3pt);
\filldraw (p15) circle (3pt);
\filldraw (p16) circle (3pt);
\filldraw (p17) circle (3pt);
\filldraw (p18) circle (3pt);
\filldraw (p19) circle (3pt);
\filldraw (p20) circle (3pt);
\filldraw (p21) circle (3pt);
\filldraw (p22) circle (3pt);
\filldraw (p23) circle (3pt);
\filldraw (p24) circle (3pt);
\filldraw (p25) circle (3pt);

\draw (p1)--(p6)--(p8)--(p1) -- (p2) -- (p3)--(p1)--(p4)--(p3) -- (p5)--(p6)--(p3);
\draw (p4)--(p6) -- (a1)-- (p1)--(p7)--(p1)--(a2)--(p6);
\draw (p8)--(p7) -- (p6) -- (p9)--(p10)--(p11)--(p9);
\draw (p11)--(p15)--(p12)--(p14)--(p11)--(p12);
\draw (p14)--(p15)--(p16)--(p12)--(p17)--(p16);
\draw (p15)--(p17)--(p18)--(p19)--(p20)--(p17);
\draw (p1)--(p17)--(p24)--(p20)--(p23)--(p22)--(p22)--(p21);
\draw (p17)--(p25)--(p20);

\draw (p1) node [below] {\hskip1mm 1};
\draw (p2) node [right] {2};
\draw (p3) node [above] {3};
\draw (p4) node [above] {\hskip2.55mm 4};
\draw (a1) node [above] {5};
\draw (a2) node [below] {6};
\draw (p7) node [below] {\hskip2.5mm 7};
\draw (p8) node [below] {8};
\draw (p5) node [left] {9};
\draw (p6) node [left] {10};
\draw (p9) node [left] {11};
\draw (p10) node [left] {12};
\draw (p11) node [below] {\hskip-2mm 13};
\draw (p12) node [above] {\hskip-2mm 14};
\draw (p13) node [above] {\hskip-5mm 15};
\draw (p14) node [below] {16};
\draw (p16) node [below] {\hskip-4.4mm 17};
\draw (p15) node [below] {\hskip2mm 18};
\draw (nn) node [above] {\hskip0.8mm 19};
\draw (p18) node [left] {20};
\draw (p19) node [above] {\hskip2mm 21};
\draw (p20) node [right] {22};
\draw (p21) node [below] {\hskip-1mm 24};
\draw (p22) node [above] {25};
\draw (p23) node [right] {26};
\draw (p24) node [below] {\hskip1mm 23};
\draw (p25) node [below] {27};
\draw (7.7,-1.5) node {13b};}
\end{tikzpicture}
\vskip-2mm\caption{Tree decomposition of a 2-connected graph}\label{fig:Tutte}\vskip-3mm
\end{figure}

The separators are $s_1=\{1,10\}$, $s_2=\{3,10\}$, $s_3=\{1,3\}$, $s_4=\{11,13\}$, $s_5=\{13,19\}$, $s_6=\{19,24\}$, $s_7=\{19,20\}$, $s_8=\{20,21\}$, $s_{9}=\{21,22\}$, $s_{10}=\{19,22\}$. Every component in the complement of a separator is one of the cuts in $\cn$. There are no slices. All separators have two components in their complement, except for $s_1$ and $s_{10}$, which are drawn white in Figure~\ref{fig:Tutte}a.

The cuts which are only $\sim$-equivalent to themselves are $\{2\}$ ($b_5$), $\{5\}$ ($b_1$), $\{6\}$ ($b_2$), $\{7,8\}$ ($b_3$), $\{9\}$ ($b_4$), $\{12\}$ ($b_7$), $\{14,15,16,17,18\}$ ($b_8$), $\{23\}$ ($b_{13}$), $\{25\}$ ($b_{10}$), $\{26\}$ ($b_{11}$), $\{27\}$ ($b_{12}$). The corresponding blocks $b_i$ (in parentheses) are the union of cut and separator. These cuts are shaded grey in Figure~\ref{fig:Tutte}a. The blocks are the leaves of the structure tree in Figure~\ref{fig:Tutte3}.

The block $b_8=\{13,14,\ldots,19\}$ is 2-inseparable (within the whole graph) but not 3-connected, but becomes 3-connected after adding an ideal edge joining 13 and 19.

The blocks $b_6=\{1,3,4,10\}$ and $b_9=\{19,20,24\}$ are shaded in dark grey in Figure~\ref{fig:Tutte}b. The corresponding $\sim$-classes consist of the cuts which are minimal with respect to containing these blocks.

The sets $b_{14}=\{ 1, 10, 11, 13, 19, 24\}$ and $b_{15}=\{ 19, 20, 21, 22 \}$ are shaded in light grey in Figure~\ref{fig:Tutte}b. They are blocks with respect to $\cn$ but not in the system of all cuts from Example~\ref{exam:finitecuts} for $\kappa=2$.
Each separator in one of these two blocks has exactly one component $C$ in its complement, such that $C\cup NC$ contains the block. These are the cuts which form the corresponding $\sim$-classes.  The $*$-complements of these cuts are arranged in a cycle.

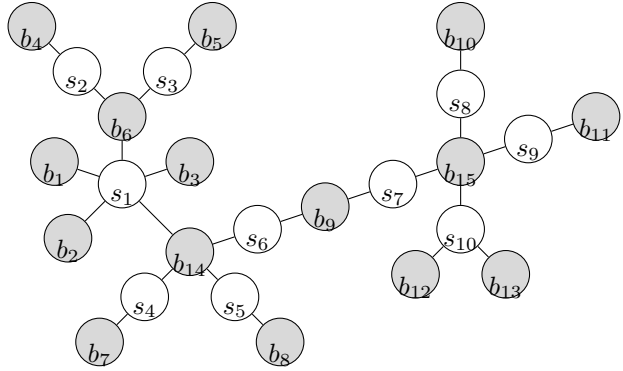
\begin{figure}[htbp]
\centering
\begin{tikzpicture}[scale=0.3]
\put(0,-5){
\node (b5) at (8,14) [shape=circle,minimum size=18pt,inner sep=0pt,draw,fill=black!15] {$b_5$}; 
\node (s1) at (6,12) [shape=circle,minimum size=18pt,inner sep=0pt,draw] {$s_3$}; 
\node (b6) at (4,10) [shape=circle,minimum size=18pt,inner sep=0pt,draw,fill=black!15] {$b_6$}; 
\node (b4) at (0,14) [shape=circle,minimum size=18pt,inner sep=0pt,draw,fill=black!15] {$b_4$}; 
\node (s3) at (2,12) [shape=circle,minimum size=18pt,inner sep=0pt,draw] {$s_2$}; 
\node (s2) at (4,7) [shape=circle,minimum size=18pt,inner sep=0pt,draw] {$s_1$}; 
\node (b14) at (7,4) [shape=circle,draw,minimum size=18pt,inner sep=0pt,fill=black!15] {$b_{14}$}; 
\node (b1) at (1,8) [shape=circle,minimum size=18pt,inner sep=0pt,draw,fill=black!15] {$b_{1}$}; 
\node (b2) at (1.6,4.6) [shape=circle,minimum size=18pt,inner sep=0pt,draw,fill=black!15] {$b_2$}; 
\node (b3) at (7,8) [shape=circle,minimum size=18pt,inner sep=0pt,draw,fill=black!15] {$b_3$}; 
\node (s4) at (5,2) [shape=circle,minimum size=18pt,inner sep=0pt,draw] {$s_4$}; 
\node (s5) at (9,2) [shape=circle,minimum size=18pt,inner sep=0pt,draw] {$s_5$}; 
\node (s8) at (10,5) [shape=circle,minimum size=18pt,inner sep=0pt,draw] {$s_6$}; 
\node (b7) at (3,0) [shape=circle,minimum size=18pt,inner sep=0pt,draw,fill=black!15] {$b_7$}; 
\node (b8) at (11,0) [shape=circle,minimum size=18pt,inner sep=0pt,draw,fill=black!15] {$b_8$}; 
\node (b9) at (13,6) [shape=circle,minimum size=18pt,inner sep=0pt,draw,fill=black!15] {$b_9$}; 
\node (s7) at (16,7) [shape=circle,minimum size=18pt,inner sep=0pt,draw] {$s_7$}; 
\node (b15) at (19,8) [shape=circle,minimum size=18pt,inner sep=0pt,draw,fill=black!15] {$b_{15}$}; 
\node (s10) at (22,9) [shape=circle,minimum size=18pt,inner sep=0pt,draw] {$s_{9}$}; 
\node (b11) at (25,10) [shape=circle,minimum size=18pt,inner sep=0pt,draw,fill=black!15] {$b_{11}$}; 
\node (s9) at (19,11) [shape=circle,minimum size=18pt,inner sep=0pt,draw] {$s_8$};
\node (b10) at (19,14) [shape=circle,minimum size=18pt,inner sep=0pt,draw,fill=black!15] {$b_{10}$}; 
\node (s6) at (19,5) [shape=circle,minimum size=18pt,inner sep=0pt,draw] {$s_{10}$}; 
\node (b12) at (17,3) [shape=circle,minimum size=18pt,inner sep=0pt,draw,fill=black!15] {$b_{12}$}; 
\node (b13) at (21,3) [shape=circle,minimum size=18pt,inner sep=0pt,draw,fill=black!15] {$b_{13}$}; 
\draw (b10)--(s9)--(b15)--(s6)--(b12);
\draw (s6)--(b13);
\draw (b14)--(s5)--(b8);
\draw (b1)--(s2)--(b3);
\draw (s2)--(b2);
\draw (b4)--(s3)--(b6)--(s1)--(b5);
\draw (b6)--(s2)--(b14);
\draw (b7)--(s4)--(b14)--(s8)--(b9)--(s7)--(b15)--(s10)--(b11);}
\end{tikzpicture}
\caption{Structure tree for the graph in Figure~\ref{fig:Tutte}}\label{fig:Tutte3}\vskip-3mm
\end{figure}
The number of cut components of a separator is the degree of the corresponding white vertex in the structure tree, see Figure~\ref{fig:Tutte3}.
The degree of a black vertex is the cardinality of the corresponding $\sim$-class.

\end{exam}

It is a consequence of our main result (Theorem~ \ref{theo:main}) that if a $G$-graph has a non-trivial cut system, then
there is a homomorphism of $G$ to the automorphism group of a tree.  Thus  there is a $G$-tree $T = T(\cn)$  associated with a nested sub-system $\cn$ of $\cc$.   The actions of groups on trees are completely described in the theory of Bass and Serre (see \cite{Dicks1989,Serre2003,Serre1980}).

If a $G$-tree has finite diameter - in particular if $T$ is finite - then there is an edge or vertex fixed by $G$.   In the case of a structure tree
all leaves correspond to blocks and so the diameter is even. If the diameter is divisible by four then $G$ fixes a block, otherwise it fixes a separator.

Finally we  note that the theory is not of any interest for finite transitive graphs.
\begin{lemm}\label{lemm:noleaves}
If $\cc$ is a thin nested $G$-invariant cut system for a $G$-transitive graph then $T(\cc)$ has no leaves.
\end{lemm}

\begin{proof}
Such a leaf would correspond to a block $B$. There is precisely one separator $S$ contained in $B$, see Corollary~\ref{coro:leaves}, and there is a vertex in $B\setminus S$. This vertex is not contained in any separator. Since $\cc$ is $G$-invariant and the graph is $G$-transitive, the same must hold for any vertex of the graph. Thus no vertex is contained in any separator, a contradiction.
\end{proof}

\begin{coro}\label{coro:finitetrans}
A finite $G$-transitive graph has no nested $G$-invariant thin cut system. In particular, the system of  Example~\ref{exam:finitecuts} is empty for finite transitive graphs.
\end{coro}

\subsection{Generalizations of Stallings' Theorem}\label{subsect:application_stalling}

If a structure tree has infinite diameter, then $G$ may induce a non-trivial action on the structure tree. We briefly describe the relevant results from Bass-Serre theory.

A group is said to \emph{split} over a subgroup $H$ if it is either a free product of two groups with amalgamation over $H$, where these two groups contain $H$ as a subgroup of index at least two, or if $G$ is an HNN-extension of some group over $H$.

An \emph{inversion of an edge} by a graph automorphism occurs if its vertices are transposed.
Suppose a group $G$ acts transitively and without inversion on the set of edges of a tree $T$. Then either the quotient $G\backslash T$ is a loop and $G$ is an HNN-extension of the stabilizer of some vertex of $T$ over the stabilizer of an incident edge. Or $G$ has two orbits on $VT$, the quotient $G\backslash T$ is  a graph with two vertices connected by an edge (called a \emph{segment}), and $G$ is a free product of the stabilizers of two adjacent vertices in $T$ with amalgamation over the stabilizer of the edge which connects them.  This decomposition is trivial if and only if the stabilizer of an edge is
the same as the stabilizer of one of its vertices and the whole group stabilizes the other vertex $v$.  If this happens then the tree $T$ has diameter 
two, with central vertex $v$.  The action is non-trivial if and only if for each edge $e \in ET$ both components of $T \setminus \{ e \}$
contain at least one edge (or equivalently, at least two vertices).  In fact if the action is non-trivial, then both components of  $T \setminus \{ e \}$
are infinite.
Thus  if $G$ acts transitively without inversion on the set of edges of $T$ then either $T$ has diameter two or $G$ splits over the stabilizer of an edge.
The latter happens if and only if for some edge $e$  both components of $T\setminus \{ e\}$ intersect the orbit of $e$.

More generally, the action (without inversion) of  a  group $G$ on a $G$-tree $T$ is non-trivial if and only if either $G$ splits over an edge stabilizer or it is a strictly ascending union
$$ G = \bigcup _n G_n, $$
where $G_1 \subset G_2\subset \dots $ is a an infinite sequence of proper subgroups of $G$ each of which stabilizes an edge of $T$.

If $G$ is a group, a Cayley graph for $G$ is a connected $G$-graph with one orbit of vertices and on which $G$ acts freely.
The edge orbits will correspond to a set of generators for $G$.  There is a locally finite Cayley graph if and only if $G$ is finitely generated.
Different locally finite  Cayley graphs of a finitely generated group are quasi-isometric. The number of ends of a locally finite graph  is a quasi-isometry invariant and hence it does not depend on the finite set of generators. Thus we define the \emph{number of ends of a finitely generated group} as the number of ends of its locally finite connected Cayley graphs.

The following was proved by Stallings in a series of papers (see \cite{Stallings1968,Stallings1970,Stallings1971}).

\begin{theo}[Stallings' Structure Theorem \cite{Dunwoody1982,Cohen1972}]
A finitely generated group has more than one end if and only if it splits over a finite subgroup.
\end{theo}

The first author proved Stallings' theorem in \cite{Dunwoody1982} by showing that the cut system of 
edge cuts, see  Example~\ref{exam:edgeends}, has a nested subsystem.  We have proved that any graph with more than one end has a nested subsystem which separates ends and is invariant under automorphisms, see Example~\ref{exam:vertexends} and Theorem~\ref{theo:optimally}. Hence we have a new proof of the main result in \cite{Dunwoody1982}. If the graph is a Cayley graph with more than one end then the tree has infinite diameter and the action is non-trivial. Thus we get  a new and relatively simple proof of Stallings' Structure Theorem. This is presented in detail in \cite{Kroen2009}.

There are different ways of generalizing Stallings' theorem. One option is to drop the assumption of $G$ being finitely generated.
Another option is to consider splittings of finitely generated groups over groups which are not necessarily finite.

There are several ways of how to define ends for non-locally finite graphs (see \cite{Kroen2001a}). The same holds for infinitely generated groups, where we have the further difficulty that without additional assumptions, the Cayley graphs are not necessarily quasi-isometric. But whatever definition one uses for the  ends of non-locally finite graphs, in locally finite graphs this definition should yield Freudenthal's end compactification for a  locally compact space (see \cite{Freudenthal1931,Freudenthal1942,Freudenthal1944}).

One way goes back to Freudenthal and Cohen \cite{Cohen1972} and says that $G$ has more than one end if there is a subset
$A$ for which $A$ and the complement $G \setminus  A$ are both infinite, and the
symmetric difference of $A$ and $Ag$ is finite for all $g$ in $G$.
It follows from the Almost Stability Theorem \cite{Dicks1989} that a group $G$ has more than one end in this sense if and only if $G$ splits over a finite subgroup or $G$ is countably infinite and locally finite. This is a generalization of Stallings' Structure Theorem, because in the finitely generated case the definition above is equivalent to all other definitions of ends of graphs and groups.
A more revealing way of stating this result follows from the Bass-Serre theory discussed above.  Thus
a group has more than one end if and only if it has a non-trivial action on a tree with finite edge stabilizers.

For a group that is not finitely generated there is no obvious way to choose a generating set to construct a Cayley graph.
Stallings' theorem can be formulated as ``A finitely generated group has a Cayley graph with more than one end if and only if it splits over a finite subgroup.'' Here we can just drop the assumption that the group is finitely generated.

\begin{theo}
A group has a Cayley graph with more than one end if and only if it splits over a finite subgroup.
\end{theo}

\begin{proof}
Suppose $G$ splits over a finite group $H$.  There are two possibilities.
Let $\mathrm{Cay}(G,S)$ denote the Cayley graph of $G$ with respect to generating set $S$. Suppose $G=G_1*_HG_2$ and $[G_i:H]\ge 2$, for $i=1,2$. If $S_i$ is a set of generators for $G_i$ then the graph $X=\mathrm{Cay}(G,S_1\cup H\cup S_2)$ has more than one end. Moreover, if $[G_1:H]=[G_2:H]=2$ then $X$ has two ends, otherwise $X$ has infinitely many ends.  If $G$ is an HNN extension $G = {G_1}*_H$, so that $G_1$ is a subgroup of
$G$ with isomorphic finite subgroups $H$ and $t^{-1}Ht$, then the Cayley graph $X =\mathrm{Cay}(G, S_1\cup \{ t \})$ has more than one end.

If $G$ has a Cayley graph $X$ with more than one end then the cut system in Example~\ref{exam:vertexends} is not-trivial and we can apply Theorem~\ref{theo:main} to get a group action of $G$ on a tree $T$. Then $G$ splits over stabilizers of elements of a cut system (i.e.\ stabilizers of the edges of $T$). The splitting is non-trivial as the graph $X$ is vertex transitive and removing any separator in the cut system will
leave at least two infinite components.
The stabilizers of a cut $A$ is finite, since it is a subgroup of the stabilizer of the finite set $NA$  and the action of $G$ on $X$ is free.
\end{proof}

A splitting of a group over a subgroup is \emph{non-trivial} if it is either an HNN-extension or in the case of a free amalgamated product, the subgroup has at least index two in each factor.
Our results provide the following generalization of Stallings' Theorem to cases when the splitting group is not finite.

\begin{theo}\label{theo:split_vertexstabilizer}
Let a group $G$ act on an infinite  graph with a $G$-invariant nested cut system. If there is a cut $C$ and a $g$ in $G$ such that $g(C)\varsubsetneq C$ or if the action of $G$ on the graph is transitive then $G$ splits non-trivially over a subgroup which has a vertex stabilizer as finite index subgroup.
\end{theo}

\begin{proof}
For any cut system $\cc$ the action of $G$ on $T(\cc )$ is without inversion of edges, because edges of $T(\cc)$ connect blocks with separators and no block can be mapped to a separator and vice versa. Thus we can apply Bass-Serre and obtain a splitting for $G$.

A cut in $\cc$ corresponds to an edge of $T(\cc)$, see Section~\ref{sec:cuts_trees}.
If $g(C\cup NC)\subset C$ then $g$ translates a double-ray in $T(\cc)$ which contains $C$. 
Thus the splitting is non-trivial. A splitting group will be the stabilizer $G_C$ of $C$. The group $G_C$ also stabilizes $NC$.  Since $NC$ is finite, $G_C$
contains a subgroup of finite index which fixes each vertex in $NC$.  Thus $G_C$ has a subgroup of finite index which fixes a vertex of $X$.
This subgroup may well be a proper subgroup of the stabilizer of this vertex.

If instead of the condition $g(C)\varsubsetneq C$ we assume that $G$ acts transitively on the graph then we know from Lemma~\ref{lemm:noleaves} that $T(\cc)$ has no leaves. Note that in general In general $G$ does not  act transitively on the blocks or the separators. But if we choose an edge $e$ of $T(\cc)$ which corresponds to a cut $C$ then there are elements $g_1,g_2$ of $G$ such that $g_1(NC)\subset C\cup NC$, $g_1(NC)\ne NC$ and $g_2(NC)\subset C^*\cup NC$, $g_2(NC)\ne NC$. Thus each of the two components of $T(\cc)-e$ contains one of the edges $g_1(e)$ and $g_2(e)$. Two of the three edges $e, g_1(e),g_2(e)$ have to point in the same direction. That is, there is an $h$ in $G$ which maps one onto the other and which translates a double ray which contains them. Thus we have the same situation as in the first case.
\end{proof}

\begin{exam}
The second part of Theorem~\ref{theo:split_vertexstabilizer} does not hold in general if we consider almost transitive graphs (i.e. there are finitely many orbits on the vertices) instead of transitive graphs. Consider the graph consisting of two (say countably) infinite complete graphs $K_1,K_2$ which share precisely one vertex $x$. This vertex is a cut vertex and admits a cut system consisting of just two cuts. There are groups acting on this graph with three (or two) orbits such that the action is transitive on $VK_1\setminus \{x\}$ and transitive on $VK_2\setminus \{x\}$ (or transitive on $VK_1\setminus \{x\}\cup VK_2\setminus \{x\}$), but the action admits no splitting.
\end{exam}

Another possible application of vertex cuts is to the Kropholler Conjecture \cite {Niblo2006}.
This arose out of work of Kropholler  \cite{Kropholler1990} on algebraic versions of the torus theorem for $3$-manifolds. 

Let $H$ be a subgroup of a group $G$. We regard $G$ as a $G$-set via the action of $G$ on the left.
A subset $A$ of a $G$-set  is called \emph{$H$-finite}  (or \emph{right $H$-finite}) if it is contained in the union of finitely many  $H$-orbits, otherwise $A$ is called \emph{$H$-infinite} (or \emph{right $H$-infinite}). Being $H$-finite is equivalent to being contained in the union of finitely many right $H$-cosets.
Analogously we call $A$ \emph{left $H$-finite} if $A$ is contained in the union of finitely many left $H$-cosets, otherwise it is called \emph{left $H$-infinite}.

\begin{conj}[Kropholler]
Let $A$ be a subset of a finitely generated group $G$ and let $H$ be a subgroup of $G$ such that $AH=A$. Let $A$ and $G\setminus A$ be $H$-infinite and let $Ag\setminus A$ be  $H$-finite, for all $g\in G$. Then $G$ admits a non-trivial splitting over a group which is commensurable with a subgroup of $H$.
\end{conj}

Theorem~\ref{theo:split_vertexstabilizer} implies the following.

\begin{coro}\label{coro:krop}
Let $A$ be a subset of a finitely generated 
group $G$ and let $H$ be a subgroup of $G$ such that $AH=A$. Let $A$ and $G\setminus A$ be left $H$-infinite and let $Ag\setminus A$ be left $H$-finite, for all $g\in G$. Then $G$ admits a non-trivial splitting over a group which is commensurable with a subgroup of $H$.
\end{coro}

\begin{proof}
Let $S$ be a finite generating set and let $X$ be the quotient of $\mathrm{Cay}(G,S)$ by the action of $H$ on the right. That is, the vertices are the left cosets of $H$ and two cosets are adjacent (as vertices of $X$) if they contain adjacent vertices
 in $\mathrm{Cay}(G,S)$. The condition $AH=A$ means that $A$ is a union of left $H$-cosets. The graph $X$ is infinite (possibly non-locally finite), $G$-transitive and the projection of $A$ is infinite and has finite boundary. Since $X$ is transitive and since there is an infinite set (the set $A$) with finite boundary and infinite complement, $X$ it has more than one end and thus a non-trivial $G$-invariant cut system. The statement now follows from Theorem~\ref{theo:split_vertexstabilizer}.
\end{proof}

If one could get a $G$-graph for the case of Kropholler's conjecture with the same properties as in Corollary~\ref{coro:krop}, then the conjecture would
follow.  One can get quite a long way in this direction.
There will be a graph $X$ in which $VX$ is the set of left cosets of $H$. The set $A$ will again determine a set $E$ of vertices of this
graph.  The set $NE$ is  contained in finitely many $H$-orbits and since $H$ fixes a vertex of $X$, $NE$ has finite diameter in $X$.
Both $E$ and $E^*$ have infinite diameter.  This implies that both $E$ and $E^*$  contain  rays.  For more details we refer to \cite[Theorem~3.5]{Kroen2008}.  
  We have not been able to show that such a graph $X$ can be constructed  in which  $NE$ is  finite, rather than just of finite diameter.
  If $G$ is the commensurizer of $H$ then one can construct $X$ so that it is locally finite.  A subset of $VX$ will then be  finite if and only if it has finite diameter.
  Thus the conjecture is true in this case. This was well known \cite {Dunwoody2000}.

\subsection{Applications in infinite graph theory.} There are several applications of vertex cuts in infinite graph theory, in particular when classical structure tree theory is used and results are generalized from locally finite to non-locally finite graphs.

In \cite{Hamann2010a} Hamann shows that an almost transitive end-transitive graph is quasi-isometric to a tree. Woess conjectured in \cite{Woess1989} that infinitely ended graphs are quasi-isometric to a tree if the stabilizer of some end acts transitively on the set of vertices. This was proved by M\"oller for locally finite graphs in \cite{Moeller1992} and generalized to non-locally finite graphs with infinitely many edge-ends in \cite{Kroen2001b}. Hamann uses vertex cuts to show that this also holds for the (more general) case of vertex ends.

A graph is called \emph{connected-homogeneous} if every isomorphism between two finite connected subgraphs extends to an automorphism of the whole graph. In
\cite{Hamann2010b} Hamann and Hundertmark use the theory of the present paper to classify connected-homogeneous digraphs and show that if their underlying undirected graph is not connected-homogeneous then they are highly arc transitive.

A graph is \emph{$k$-CS-transitive} if for any two connected isomorphic subgraphs of order $k$ there is an automorphism between them which extends to the whole graph.
Hamann and Pott \cite{Hamann2009} use vertex cuts in order classify $k$-CS-transitive graphs for all $k$ and non-locally finite distance transitive graphs which have more than one end.

\section*{Acknowledgements}

The authors first met at the conference on Totally Disconnected Groups, Graphs and Geometry at Blaubeuren (Germany) in May 2007. They are very grateful to the organizers for inviting them to the meeting.
This project grew from a problem raised at a problem session at that conference.

The first author attended lectures by P.~Papasoglu at Heriot Watt University in September 2008 in which he
presented a proof of Stallings' Theorem, arising from  new work of himself and A.~Evangelidou \cite{Evang2010},
which uses some similar techniques to that of this paper.

The first version of this paper was posted as arXiv:0905.0064 in May, 2009.�

R.\ Diestel has organized a seminar of the Univ.\ of Hamburg on the present paper which took place in B\"usum (Germany) from 29 March to 1 April 2010. The authors wish to thank the participants E.~Bu\ss, J.~Carmesin, J.-O.~Fr\"ohlich, M.~Hamann, K.~Heuer, J.~Hiob, F.\ Hundertmark, M.~Kriesell, M.~M\"uller, H.~Oberkampf, J.~Pott, T.S.~R\"uhmann and G.~Zetzsche who found several typos and mistakes and made numerous valuable contributions.
M.~Hamann found a mistake in a theorem in one of the previous versions of the paper. M.~Kriesell has pointed out that one of the axioms of cut-systems in a previous version of the paper was implied by the other axioms.

R. Diestel's great interest in the subject has been such that he not only provided the authors, in 2010, with a good amount of further suggestions and comments. Subsequently, he elaborated together with his co-workers his own, similar, version of an extension of Tutte's tree decomposition to finite graphs with arbitrary connectivity \cite {CD}.

We thank Wolfgang Woess for his valuable suggestions, and his help and encouragement.

\bigskip
\halign{ # \hfil \qquad \qquad \qquad \qquad & # \cr

%\affiliationone{% in this example, two authors share an institution
  M.J. Dunwoody & B.Kr\"on \cr
  %  Faculty of Mathematical Studies,
    University of Southampton & University of Vienna \cr
   % School of Mathematics & Faculty of Mathematics \cr
 %   Southampton, SO17 1BJ & A-1090 Vienna\cr
   % England & Austria\cr
     M.J.Dunwoody@soton.ac.uk &  bernhard.kroen@univie.ac.at \cr
     }
     \end {document}